\documentclass[11pt,leqno]{article}

\usepackage{epsfig}

\usepackage{hyperref}

 \usepackage{amssymb}
 \usepackage{amsmath}

\newenvironment{proof}{\paragraph{\underline{Proof} :}}{\hfill$\square$}

  \newenvironment{proofth}[1]{\paragraph{\underline{#1}} :}{\hfill$\square$}

\usepackage{mathtools}
\mathtoolsset{showonlyrefs}

    \newtheorem{theorem}{Theorem}

    \newtheorem{lemma}{Lemma}
    \newtheorem{remark}{Remark}
    \newtheorem{corollary}{Corollary}
    
    \newtheorem{definition}{Definition}
    
        \newtheorem{aptheorem}{Theorem}
        \newtheorem{aplemma}{Lemma}
                \newtheorem{apremark}{Remark}

\newcommand{\eps}{\varepsilon}

\newcommand{\A}{\mathcal{A}}

\newcommand{\D}{\mathcal{D}}
\newcommand{\X}{\mathcal{X}}
\newcommand{\I}{\mathcal{I}}

\newcommand{\N}{\mathcal{N}}

    \newcommand\CC{\hbox{C\kern -.58em {\raise .54ex \hbox
    {$\scriptscriptstyle |$}}
    \kern-.55em {\raise .53ex \hbox{$\scriptscriptstyle |$}} }}
      \newcommand\qd{\hfill$\sqcap\kern-8.0pt\hbox{$\sqcup$}$}
    \newcommand\NN{\hbox{I\kern-.2em\hbox{N}}}
       \newcommand\nn{\hbox{I\kern-.2em\hbox{N}}}
    \newcommand\RR{I\!\!R}
    \newcommand\sRR{{\sl \hbox{I\kern-.2em\hbox{R}}}}
    \newcommand\QQ{\hbox{I\kern-.53em\hbox{p}}}
     
    \newcommand\J{\displaystyle {\cal J}}
    
  \newcommand{\argmin}{\hbox{argmin}}     

 \everymath{\displaystyle}

\usepackage{stackengine}
\usepackage{scalerel}
\usepackage{xcolor,amssymb}

\title{Minimum   Flow  Steepest Descent     Approach   for  \\   Nonlinear PDE   }
\author{Noureddine Igbida  \thanks{  Institut de recherche XLIM, UMR-CNRS 7252, Faculté des Sciences et Techniques, Université de Limoges, France. Email:  noureddine.igbida@unilim.fr}}

\date{\today}

\begin{document}
 \maketitle

 \begin{abstract}
 	 
 	 	This paper presents a   minimum flow approach applicable to a wide range of doubly nonlinear diffusion problems. We introduce a minimum flow steepest descent algorithm that seeks an optimal traffic flow by minimizing an internal energy function, while incorporating a specific minimum flow constraint for the transition work. This flexible framework surpasses traditional methods by generalizing (among other things) the established $H^{-1}$-theory for linear diffusion to the nonlinear setting. It offers distinct advantages in handling diverse applications, leveraging both the intrinsic internal energy and the inherent traffic flow concepts. 	
 	 	 The approach demonstrably tackles various applications, including fluid flow in porous media and many other scenarios  like the Stefan/Hele-Shaw problem. Notably, it can handle diverse differential operators, encompassing even nonlinear ones like the $p-$Laplacian and Leray-Lions operators.  
 	 	  Furthermore, it allows for the simultaneous handling convection/transport phenomena and  complex boundary conditions within both the energy and transition work functions, contrasting favorably with other steepest descent algorithm  like the JKO scheme in Wasserstein spaces. 
 	 	   The paper delves into details of this comparison and offers additional theoretical insights involving $p'-$curve and metric gradient flow in Sobolev dual spaces.

 \end{abstract}

 \tableofcontents

 
\section{Introduction and preliminaries}
\setcounter{equation}{0}

\subsection{Introduction}
Dynamic systems generally aim to minimize their internal energy over time. By examining the dynamics at a discrete level, we can develop strategies that systematically guide the system towards the state of minimal energy using a step-by-step approach called steepest descent. One efficient approach involves managing the process of proximal energy minimization. This entails a dynamic procedure that continuously updates the system's state, aiming to decrease its internal energy while simultaneously minimizing the effort needed to transition from the current state to the subsequent one.

This approach has proven to be a robust tool for both describing continuous dynamics and establishing new dynamic systems that utilize their internal energy and the effort required to transition between states. Additionally, this approach has proven to be a valuable tool for devising dynamic strategies that aim to achieve energy-minimized states in various applications, including control theory, machine learning, and physics.  However, most existing steepest descent algorithms    often assess transition effort based on topological properties like gradient flow within specific spaces (e.g., Hilbert, Banach, metric Wasserstein), which may  limit their flexibility. Our approach goes through the so-called minimum flow problem and  offers a way to overcome this limitation, enabling steepest descent to minimize internal energy in a broader class of nonlinear PDEs like 
\begin{equation}\label{E0}
	\partial_t \rho-\nabla \cdot (\phi  )= f,\quad \phi  = \phi  (x,\nabla \eta),\:   \eta= \eta(x,\rho)\quad \hbox{ for } (t,x)\in Q:=(0,T)\times \Omega,   
\end{equation}
in  an open and bounded  regular domain  $\Omega$,  subject to  non-homogeneous boundary conditions of Dirichlet, Neumann, or mixed type.  Here $T>0$  is an horizon time and $f$  is source term.  The applications $A\in \RR^N\to \phi  (x,A)\in \RR^N$  and  $r\in \RR\to \eta(x,r)\in \RR$  are given maps, possibly multivalued,  and assumed to be maximal monotone.  One sees that,  when the mobility function $\phi $  is linear, specifically $\phi  (x,\nabla \eta)=\nabla \eta,$  the resulting diffusion process is characterized by linear diffusion. For nonlinear mobility functions like $\phi  (x,\nabla \eta)=\vert \nabla \eta\vert^{p-2}\nabla \eta,$  the diffusion process exhibits $p$-Laplacian characteristics.  

In addition to its theoretical value in proving existence and providing numerical approximations, our approach offers a meaningful interpretation that connects the solution to the system's internal energy and the various parameters, such as the source term and boundary conditions. It also establishes a direct link between the Euler-Implicit scheme, commonly used in the Crandall-Liggett $L^1-$ theory, and a steepest descent algorithm that aims to minimize a proximal internal energy. This connection is possible despite the fact that the primary operator driving the dynamic is not necessarily a sub-differential operator in any  usual Banach space ($L^p,$  Sobolev or its dual space).  In certain cases, our approach enables to restore   the sub-differential operator process through the use of its metric version in dual Sobolev space. 
For homogeneous problems (homogeneous $\phi $  and boundary conditions), the approach can be simply linked to classical Hilbert gradient flow theory (cf. \cite{Br,Barbu}) in $H^{-1}$ for linear $\phi $ ; i.e., $\phi  (x,\nabla \eta)=  \nabla \eta.$  Furthermore, for general $\phi $ of $p-$Laplacian type, it can be connected to metric gradient flow and $p'-$curve of maximal slope theory in the metric space $W^{-1,p'}$ 

%

\medskip  
 We show, as well, 
that the  approach provides a valuable tool for understanding and solving systems with dynamics that can be decomposed into two processes : a forcing process  connected for instance to a given transport vector field, and  minimizing internal energy process.  More precisely ;  we show how one can use the approach to handle nonlinear diffusion-transport equation of the form 
\begin{equation}\label{EDT}
	\partial_t \rho-\nabla \cdot (\phi  +\rho \: V)= f,\quad \phi  = \phi  (x,\nabla \eta),\:   \eta= \eta(x,\rho)\quad \hbox{ in }Q ,   
\end{equation} 
where $V=V(t,x)$ is a given transport fields.   Roughly speaking,  the dynamics of the system here can be decomposed into two processes : the transport process describes how the system moves under the action of the forcing vector field $V,$ while the minimization  energy process describes how the system tends to reach a stable equilibrium.   Degenerate nonlinear parabolic problems like \eqref{EDT}  received a lot of attention in the past couple of
decades. For a non-exhaustive list of classical works on this subject we refer to \cite{AltLu,Car,Vbook,Peletier} and the references therein.  The approach we present provides a way to understand how these two processes interact to produce the system's dynamics. In particular, it shows that the solution of the system is the state that minimizes the system's internal energy subject to the constraints imposed by the transport process and the boundary conditions.  This interpretation has several implications. First, it provides a way to understand how the system's dynamics are affected by the different parameters. For example, the source term can be interpreted as a force that acts on the system, and the boundary conditions can be interpreted as constraints that limit the system's motion.  Second, the interpretation can be used to design new algorithms for solving the system's dynamics from theoretical and numerical point of view.

  Let us recall that the approach using the JKO algorithm in Wasserstein spaces has already been successfully applied to partially address this type of question (cf. \cite{Vbook,AGS,Stbook,KM}).  
  However, this theory's limitations to positive data, mass conservation, specific internal energies, and its technical complexity did not allow to broaden the scope of applications.     The approach we present here draws inspiration from the JKO algorithm and the minimal flow problem (cf. \cite{Beckmann}). It  allows to cover a wide range of applications that go beyond $H^{-1}$-gradient flow framework  while keeping the Wasserstein spirit for a large field of applications.

  \subsection{Preliminaries}  
  More formally,   let $X$ be a state space and $\mathcal E(\rho)$ be the internal energy of a given system  at state $\rho\in X.$   At each time $t>0,$  the dynamic system tends simultaneously  to decrease its internal energy and to minimize the work required to move from state $\rho(t)$ to state $\rho(t+h).$  This can be achieved by solving the following optimization problem (proximal energy) at each time step:
  $$\min_{\rho \in X}  \Big\{  \mathcal E(\rho) + W_h(\rho(t), \rho) \Big\} ,$$ 
  where $W_h(\rho(t),\rho)$ is a measure of the work required to move from state $\rho(t)$ to state $\rho,$ throughout small time step $h>0.$  The solution to this optimization problem $\rho(t+h) $ is the next state of the system.   To the point, one can  consider the  $\tau-$time steps defined by  $t_0 = 0 < t_1<\cdots < t_n <T$, and    the sequence of piecewise constant curves
  $$\rho_\tau =\sum_{i=1}^n\rho_i   \:  \chi_{]t_{i-1}  ,t_{i}]}  , $$
  where  \begin{equation}\label{steepestdescent0}
  	\rho_i= \argmin_{\rho \in X} \Big\{  \mathcal E  ( \rho) + W_\tau(\rho_{i-1}, \rho)\Big\} ,\quad \hbox{   for }i=1,...n, \hbox{ with }\rho_0=\rho(0).
  \end{equation}
  Then, we expect  the limit as $\tau\to 0$  of $\rho_\tau$ converges to  the solution of the evolution problem  \eqref{E0}.   
  Working within the framework of   Hilbert  space $(X,\Vert -\Vert),$ the general theory suggest to connect $W(\rho(t), \rho)$ to $\Vert\rho(t)- \rho \Vert^2/2\tau $   to  obtain a representation of the solution through  the gradient flow equation:
  \begin{equation}\label{evol0}
  	\left\{ \begin{array}{ll} \rho_t(t)    +\partial \mathcal E(\rho(t)) \ni 0\quad \hbox{ in }(0,T) \\  \\ 
  		x(0)\hbox{ is given},  \end{array}\right. 
  \end{equation} 
  where $\partial \mathcal E(\rho)$ denote sthe sub-differential (possibly multi-valued map) of the functional $\rho\in X\to \mathcal E(\rho)\in (-\infty,\infty].$   The expression \eqref{steepestdescent0} can be related to   both to the resolvent associated with the operator $\partial \mathcal E $  and Euler-Implicit discretisation of \eqref{evol0}. The   evolution problem  in turn  is   a gradient flow  in a Hilbert  space  which describes the evolution of a continuous curve $t\to \rho(t)\in X$ that follows the direction of steepest descent of the  functional $\mathcal E$ (cf. \cite{Br}).  One can consider a more general situation where  $(X,d)$ is metric space and  $W_h$  is build on the distance $d$. In this case, the expression \eqref{steepestdescent0} reveals   the minimizing movement scheme à la  De Giorgi (see \cite{AGS}). 
  The approach resulted in a broad theory of gradient flows in a metric space where the derivatives $  \rho_t  $ and $ \partial_\rho \mathcal E(\rho)$ need to be interpreted in an appropriate   way in $(X,d)$ (cf. the book \cite{AGS}).

  \medskip  
  
  For instance, the porous medium-like equation
  \begin{equation}\label{PME} 
  	\partial_t \rho -\Delta \rho^m = 0 
  \end{equation}
  in a bounded domain  arises in various applications, especially in biology, to represent the evolution  of species as they strive to minimize their internal energy $\mathcal E(\rho),$ which depends inherently on the density $\rho.$  Since the works \cite{OttoPME} and \cite{JKO}, it has been established that the PDE \eqref{PME} subject to homogeneous Neuman boundary condition can  be derived by employing $\mathcal E(\rho)=\frac{1}{m-1}\int \rho^m \: dx$  in the context of $W_2-$Wasserstein distance by setting $W_\tau (\rho,\tilde \rho)= \frac{1}{2\tau} \mathcal W_2( \rho,\tilde \rho)^2$ (see Section \ref{SOverviewW} for the definition of $\mathcal W_2$).   
  It is well known that it can be also derived by utilizing the steepest descent algorithm for the internal energy $\mathcal E(\rho)=\frac{1}{m+1}\int \rho^{m+1} \: dx$  in the context of $H^{-1}-$norm.  This is achieved by setting $W_\tau (\rho,\tilde \rho)= \frac{1}{2\tau} \Vert \rho-\tilde \rho\Vert _{H^{-1}}^2$  (cf. \cite{Barbu}).   Yet, one must carefully define the $H^{-1}-$norm to adequately address boundary conditions (see Section \ref{SH1}).  These concepts can be expanded to accommodate some monotone nonlinearity $\eta(r)$ instead of $r^m,$ both for $W_2-$Wasserstein distance (cf.  \cite{KM} and the refs therein) and also $H^{-1}-$norm (cf. \cite{Dam,DK,Kenmochi}) frameworks.   Owing to the significant advancements in optimal mass transportation theory, the work of \cite{Agueh}  permits handling doubly nonlinear PDE
  $$\partial_t \rho -\nabla \cdot \phi (\nabla \eta(\rho)) = 0$$
  in the context of steepest descent algorithm and gradient flow with Wassertsein distance $W_c$, where $c$ is closely linked to the application $\phi $  assumed in these works to be  regular and satisfying the classical assumption of ellipticity. In particular, this enables providing an intriguing interpretation via steepest descent algorithm in Wasserstein space for porous medium-like equations governed by $p-$Laplacian nonlinear diffusion operator  instead of the Laplacian.  
  However, as far as we are aware, the extension of the concept to more general operators  as in \eqref{E0} remains unclear. Typical example may be given by   space-dependent $\phi $  and/or more general boundary conditions like Dirichlet, Neumann, or mixed homogeneous/non-homogeneous boundary conditions. The crux of the issue lies in the absence of a gradient flow framework to handle these types of extensions. Neither the $H^{-1}$ approach nor other  possible extension to Sobolev dual approach appear to be capable of addressing these issues, even for the $p-$Laplacian operator. Our objective is to develop an alternative steepest descent-type approach for an internal energy that generalizes the $H^{-1}-$approach, in the spirit of Wasserstein distance. The interest of this approach stems from its physical interpretation of the dynamics studied and its interrelation with the system's internal energy and boundary conditions, even if the evolution PDE may not be related directly to a gradient flow dynamic.

  \medskip 
  Coming back to the class of PDE \eqref{E0}, remember that $\rho\geq 0$ represents in many situations an unknown density of a fluid or a population, but $\rho$   could also be a sign-changing solution. In this case, the equation models the evolution of two types of materials/species  whose densities are represented by the positive and negative parts of $\rho$, respectively. The equation   highlights a second-order term to include linear or nonlinear diffusion phenomena through the term $\nabla \cdot  \phi  (x,\nabla \eta) .$  However, as we can see, the diffusion does not act directly on $\rho$, but on the parameter $\eta,$ which depends on $\rho$. In fluid mechanics, $\eta $ can be related to pressure, but in certain situations, it can simply appear as a corrective potential to adjust or reorient the transport phenomena of the model like in crowed motion (cf. \cite{MRS1}, \cite{MRS2}, \cite{MRSV}, \cite{EIJ}).  The relationship between $\rho$ and $\eta $, as well as the relationship between $\phi $  and $\nabla \eta,$ depends strongly on the physical phenomenon being modeled. This relationship can be linear, nonlinear, or multivalued, but it must always be monotone to ensure convexity  of the energy and well-posedness. 
  Typical example for the  applications (possibly multi-valued)  $(x,A)\to \phi (x,A)\in \RR^N$ and $(x,r)\to \eta(x,r)\in \RR^N$  includes the following problems  :   
  \begin{itemize}
  	\item Parabolic equation governed by a Leray-Lions type operator :  $\phi (x,A)$ is a Caratheodory applications satisfying classical ellipticity assumptions. 
  	\item Parabolic p-Laplacian equation : $\phi  (x,A)=k(x)\vert A\vert^{p-2}A ,$   	where $0<k$ is a given continuous function. 
  	\item Porous-Medium equation : $\eta(x,\rho)=\rho^m $     
  	\item Stefan problem : $\eta(x,r)= (r-a_1)^+-(a_2+r)^- ,$ where $a_1$ and $a_2$ are two given non-negative functions.  
  	\item Hele-Shaw problem $\eta(x,r)= a(x)\: \hbox{Sign}^{-1}(r) ,$ where  $\hbox{Sign}$ denotes the usual sign graph, and  $a$ is a given positive function. 
  \end{itemize}

  Our aim in this paper    is to explore this class of PDE through a “steepest descent” of an intrinsic internal energy functional we associate with the system   against a minimum flow  work.  To the best of our knowledge, the proposed approach is a new and innovative perspective that is comparable to, but not necessarily identical to, the optimal mass transport approach (as we will see). Additionally, one of the key features of the proposed approach is its ability to handle changing sign solutions and a variety of boundary conditions, including mixed non-homogeneous and Robin boundary conditions.    It is important to note that Wasserstein distances capture the movement of non-negative masses while maintaining mass conservation, making them suitable for describing the movement of particles within a given domain. However, the translation of Dirichlet boundary conditions into Wasserstein gradient flows remains a challenge (cf. \cite{FG}).

\section{Plan of the paper}

This paper focuses on an approach for solving evolution problems of the type \eqref{E0} and  \eqref{EDT} using the steepest descent method through minimum flow process.
  In Section \ref{SOverview}, we provide a brief overview of this algorithm  in the Wasserstein and $H^{-1}-$frameworks, followed by a formal description of our proposed approach.  In Section \ref{SAssumptions}, we introduce     assumptions and review fundamental concepts related to Sobolev functions, their traces, and normal traces of vector-valued Lebesgue functions with Lebesgue-integrable divergence. Section \ref{SProximal} delves into a comprehensive analysis of the proximal energy minimization problem, which serves as the foundation for our steepest descent algorithm. Section \ref{SSteepest} demonstrates how to employ the proximal energy minimization problem and steepest descent algorithm to investigate evolution problems of the type \eqref{E0} and  \eqref{EDT}.    Section \ref{SRemarks} offers general remarks, perspectives, and open questions related to the proposed approach.  At last, 
 Appendix section presents some technical lemmas which we use in the paper.

\section{Overview  : Wasserstein distance vs $H^{-1}-$norm}\label{SOverview}
\setcounter{equation}{0}
 
The aim of this section is to remind the reader  formally the main idea of   steepest descent algorithms for nonlinear PDE  in the   context of $W_2-$Wasserstein distance and   $H^{-1}-$Hilbert space .

\subsection{Proximity through $W_2-$Wasserstein distance}\label{SOverviewW}

  Since the work of \cite{JKO}, the approach  of steepest descent algorithms  is well known and well used in the study of linear and nonlinear diffusion-convection  phenomena. Here we focus the description on  PDE of the type \eqref{PME}.   
  
  	Remember, that $W_2-$Wasserstein distance between two Radon measures $\mu_1$ and $\mu_2$ with equal masses 	is defined as
  		$$	  \mathcal {W}_2(\mu, \nu) =  \left(\inf_{\gamma\in \mathcal{P}(\Omega\times \Omega)  } \Big\{ \int \|x-y\|^2d\gamma(x,y)\: :\:    \pi_x\#\gamma= \mu,\:  \pi_y\#  \gamma =\nu \Big\} \right)^{1/2} $$
  		The $W_2-$Wasserstein space 	is the metric space of probability measures	endowed with the  metric $\mathcal{W}_2.$   The $W_2-$Wasserstein distance between probability measures $\mu\in   \mathcal{P}(\Omega)$ 	and $\nu\in  \mathcal{P}(\Omega)$ is given also by  
\begin{equation}\label{BBF}
	  	\mathcal  {W}_2(\mu, \nu)^2= \inf\Big\{  \int_0^1\!\! \int \vert V(t,x)\vert^2\: d\rho (t,x) \: :\:  (\rho,V)\in \A (\mu, \nu)   \Big\} , 
\end{equation} 
  		where $$  \A (\mu, \nu)  := \Big\{ (\rho,V) \: :\: \rho \in AC([0,1]; \mathcal{P}_2(\Omega)),\:  V\in L^2_\rho([0,1]\times \Omega),\:   \rho(0)=\mu,\: \rho(1)=\nu, $$
  		$$\hbox{ and }\partial_t \rho + \nabla\cdot( V\: \rho) = 0\hbox{ in }[0,1]\times \overline \Omega  \Big\} .$$
  Let us consider an internal energy  functional $ \mathcal E\::\:  \mathcal{P}_2(\Omega)  \to [0,\infty]$    
  	given by 
  		\begin{equation}\label{internal0}
  		 \mathcal E ( \rho)=\left\{ \begin{array}{ll}
  		 	\int_\Omega f(\rho)\: dx   \quad & \hbox{  if }f( \rho)\in L^1(\Omega)\\   \\ 
  		 	+\infty & \hbox{ otherwise}
  		 	\end{array}\right. 
  	\end{equation} 
  	where $f$ is a given convex function assumed here  to be regular enough. 
  	The steepest descent algorithm associated with $\mathcal E$ through Wasserstein distance $	\mathcal  {W}_2,$ usually known as the JKO scheme since the seminal work \cite{JKO}, is given by 	    the proximal algorithm 
  		$$	\rho_{i+1} =  \hbox{argmin}_{\rho} \left\{  \frac{1}{2\tau} 	\mathcal  {W}_2(  \rho,\rho  )^2 +    \mathcal E( \rho) \right\}    \quad \quad \hbox{  for }, i=0,1,...n$$   
  		Then, taking an horizon time $T>0,$  $\tau-$time discretization    $t_0 = 0 < t_1<\cdots < t_n <T$  and letting $\tau\to 0$ in the approximation   
  		$$\rho_\tau =\sum_{i}\rho_i   \:  \chi_{[t _{i} ,t_{i+1}[} ,  $$
  		one covers the solution of the evolution problem 
\begin{equation}\label{eqwh}
	  	   \left\{ \begin{array}{ll}  \partial_t \rho  -\nabla \cdot  ( \rho\:  \nabla   f'(\rho) ) =0\quad \hbox{ in }Q:=(0,T)\times \Omega  \\  \\ 
  		 \rho\:  \nabla   f'(\rho)  \cdot \nu =0\quad \hbox{ on }\Sigma:=(0,T)\times \partial \Omega  \\  \\  
  			\rho(0)=\rho_0 .\end{array}\right.
\end{equation}   
  		For   classical results of this approach and related issues we refer the reader to the books \cite{AGS,Stbook} for a thoroughgoing  results as well as a corresponding 	literature

  		\subsection{Proximity through $H^{-1}-$norm}\label{SH1}
  	To keep the same context as the Wasserstein case, which conserves the total mass, we    present here the core concept of   an appropriate framework for the $H^{-1}$ approach (for more in-depth discussions, we  refer to  \cite{Kenmochi,DK}).   The conservation of mass involves considering the quotient vector space 
  		$$H:= H^1(\Omega)/\RR.$$  
  		That is the quotient space associated with the binary relation defined in $H^1(\Omega),$ by 
  		$$u_1\mathcal R u_2\quad \Longleftrightarrow  \quad u_1-u_2=c\hbox{   a.e in }\Omega, \hbox{ with }c\in \RR.  $$
  		 To avoid any confusion, we will identify the vectors $\overline u\in H,$  with their corresponding elements $u\in \overline u.$  
  		The space $H  $ is Hilbert space when equipped with the inner product 
  		$$a(u,v)=\int_\Omega \nabla u\cdot \nabla v\: dx,\quad \hbox{ for any }u,v\in H , $$ 
  		and its  associate norm 
  		$$\Vert u\Vert_{H^{-1}}=\Vert \nabla u\Vert_{L^2(\Omega)},\quad \hbox{ for any }u\in H.  $$ 
  		Then, let us consider $H^{-1},$ the  dual space (topological) of $H.$  Let us denote by $J$ the duality map  of $H$ and use the notation  $\langle \mu,z\rangle $ for the value of  $\mu\in  H^{-1}$  at $u\in H. $  We denote by $\Vert \: \Vert_{H^{-1}}$ the dual norm in $H^{-1}$ induced by $\Vert \: \Vert_{H }$ in $H.$ The duality map  $J$ may be defined  by $  J\: :\:  H^{-1} \to H $ and $J(\mu)=u,$ where $u$ is a solution of the Neuman boundary problem governed by the Laplace operator 
  		$$\left\{ \begin{array}{ll} 
  			-\Delta u=\mu \quad & \hbox{ in }\Omega\\ 
  			\partial_\eta u=0 &\hbox{ on }\partial \Omega.\\  \end{array}\right. $$ 
  		Moreover, the duality bracket  $\langle \: ,\: \rangle $  is given by 
  		$$ \langle \mu,u\rangle    =\int_\Omega \nabla J(\mu)\cdot \nabla u\: dx,\quad \hbox{ for any }\mu\in H^{-1}\hbox{ and }u\in H. $$ 
  		In particular this implies that 
  		$$\Vert \mu \Vert_{H^{-1}} =\Vert \nabla J(\mu)\Vert_{L^2(\Omega)}.   $$
 The steepest descent algorithm associated with a given  internal energy  $ \mathcal E$  through $H^{-1}-$norms appears through the usual Euler-Implicit time   discretization associated with the evolution problem 
 \begin{equation}\label{evolsubdiff}
 	\frac{du}{dt}+ \partial \mathcal E (u)\ni 0\quad  \hbox{ in }(0,T),
 \end{equation}
 where $ \partial  \mathcal E$ is given by the usual definition of sub-differential in Hilbert space which is $H^{-1}$ here. 
 Indeed, taking a $\tau-$time discretization    $t_0 = 0 < t_1<\cdots < t_n <T$ the Euler-Implicit time   discretization  aims to solve the stationary problem 
 $$u_{i+1}+\tau \partial \mathcal E (u_{i+1}) \nu u_{i},\quad \hbox{ for each }i=0,...n,$$
 where $u_0$ is the initial data associated with the evolution problem \eqref{evolsubdiff}.  In particular, $u_{i+1}$ may be given by 
 	$$	u_{i+1} =  \hbox{argmin}_{\rho} \left\{  \frac{1}{2\tau } \Vert u_{i}-u\Vert_{H^{-1}} ^2 +     \mathcal E(u) \right\}    \quad \quad \hbox{  for }, i=0,1,...n$$ 
   	Then,     letting $\tau\to 0$ in the approximation   
  	 	  $$u_\tau =\sum_{i}u_{i}   \:  \chi_{[t_{i}  ,t_{i+1}[}  , $$
  		one covers the solution of the evolution problem  \eqref{evolsubdiff}.   For the internal energy    \eqref{internal0}, the associate PDE  read 
  		\begin{equation}\label{eqH}
  			\left\{ \begin{array}{ll}  \partial_t u  -\Delta    f'(u)  =0\quad \hbox{ in }Q \\  \\ 
  			  \nabla   f'(u)  \cdot \nu =0\quad \hbox{ on }\Sigma  \\  \\  
  				u(0)=u_0 \end{array}\right.
  		\end{equation}


 \begin{remark}
In the  non-conservatif case, the usual $H^{-1}-$approach   consists to take simply the dual space of $H^1_0(\Omega)$ equipped with its usual norm. In this case, the associate nonlinear PDE is subject to homogeneous    Dirichlet boundary condition (\cite{Barbu}, one can see also remark \ref{RemH1} in Section \ref{SProximal} ). 
 	\end{remark}

   	\subsection{Minimum flow concept and related steepest descent}

The divergence differential operator can be used to model the movement of particles between regions with different concentrations. The equation
\begin{equation}\label{div}
	-\nabla \cdot \phi  = \mu - \nu \quad \text{in} \ \Omega
\end{equation}
characterizes the traffic flow $\phi $  between two distributions represented by a source $\mu$ and a target $\nu$ within an open bounded domain $\Omega.$ The minimum flow problem, originally formulated by Beckman (cf. \cite{Beckmann}), involves finding an optimal traffic flow by minimizing a specified cost function on $\phi .$  Inspired by the concept presented in \cite{Beckmann}, we propose a generalization of the minimum flow problem that can accommodate the study   of relevant practical scenarios for \eqref{EDT}.  For each given traffic flow, there is a corresponding traffic cost involving two outlays:  $F(x,   \phi (x) )$ in $\Omega$ and  $G(x, -\phi (x)\cdot \nu(x))$ on $\Gamma_D$ a specific portion of the boundary. Here,   $F\: :\:  \Omega \times \mathbb{R}^N \to \mathbb{R} $ and 	 $G\: :\:   \partial \Omega \times \mathbb{R}  \to \mathbb{R} $   are given functions. In this context, the optimal traffic cost is defined as the minimum total cost of the traffic, which corresponds to the minimum flow of the traffic cost 
\begin{equation}\label{bec0}
 \int_\Omega F(., \phi )\: dx + \int_{\Gamma_D} G(., -\phi \cdot \nu) \: d\sigma  .
\end{equation}
The goal is to find the traffic scheme that results in the lowest overall cost, taking into account the cost functions $F$ and $G$ applied to the flow quantities $\phi $  and $ \phi \cdot \nu$ at each location in the domain $\Omega$ and on  the boundary $  \Gamma_D,$ respectively.  

In other words, we define an intrinsic quantity, denoted for the moment by  $I(\mu,\nu),$  for each distribution of mass $\mu$ and $\nu$. This quantity is defined as the minimum of the value \eqref{bec0} among all vector valued flux $\phi $ satisfying the balance equation \eqref{div}. It characterizes the movement of mass from $\mu$ to $\nu.$  The concept of $I(\mu,\nu)  $ is reminiscent of optimal mass transportation, where the goal is to find the most efficient way to transport mass from one distribution to another while minimizing the total traffic  cost. Similarly, in our context, $I(\mu,\nu) $  can be interpreted as a transfer fee that captures the essential characteristics of the mass transfer from $\mu$ to $\nu$, providing insight into the optimal traffic  process.  	 Observe that equation \eqref{div} can alternatively be viewed as a   transport equation governing the transfer of mass $\mu$ into $\nu.$   This can be seen by introducing the following variables:
$$ V=- \phi  /\rho ,$$ 
where $$\rho(t) =  t\nu+(1-t)\mu ,\quad \hbox{ for any }t\in [0,1].$$
By adopting this formulation,   \eqref{div}   shares formally the same objective as a transport equation, 
$$\left\{  \begin{array}{ll} 
	\partial_t \rho + \nabla \cdot( \rho\: V)=0\quad &\hbox{ in } (0,1)\times \Omega\\  \\
	\rho(0)=\mu, \quad \rho(1)=\nu, \end{array}\right.  $$
which is to describe   the movement of a quantity, represented by $\rho$, across a given domain through the vector field  $V.$   Yet, this is just formal since the velocity $V$ may be here infinite in general.  See in this case that the total cost \eqref{bec0}  may be written as 
\begin{equation} \label{exprTO}
	\int_0^1\!\! \int_\Omega F(x, \rho(t,x)\: V(t,x))\:   d xdt  +	\int_0^1\!\! \int_{\Gamma_D} G(x,-\rho(t,x) V(t,x)\cdot \nu) \:  d\sigma  d t .
\end{equation}
While the two formulas \eqref{exprTO} and \eqref{BBF} are not the same, it is noteworthy that the expressions \eqref{exprTO} bears resemblance to the Wasserstein distance expression \eqref{BBF}. This similarity suggests that the approach we are developing here shares some of the characteristics of optimal transport, which is already more or less established for the $H^{-1}-$approach (cf. \cite{Vbook}). 
This observation motivates the use of the quantity $I$ to measure transition work in a steepest descent algorithm for minimizing internal energy. Another compelling reason to incorporate $I$ into the proximal energy minimization algorithm is that the transfer phenomenon we are addressing arises directly and instantaneously from diffusion phenomena with different boundary conditions. Indeed, formal duality arguments can demonstrate that the Euler-Lagrange equation associated with $I(\mu,\nu)$  is closely related to   elliptic problem : 
 \begin{equation}  \label{Eformal}
 	\left\{ \begin{array}{ll}
 		\left. \begin{array}{l}    	-\nabla \cdot \phi = \mu-\nu   \\  \\  
 			\phi \in \partial _\xi F^*(x,\nabla u)     \end{array} \right\}   \quad & \hbox{ in }\Omega 
 		\\  \\ 
 		\mu_{/\Gamma}-	\phi   \cdot \nu \in \partial G^*(x,u) &  \hbox{ on }\Gamma_D  	\\  \\ 
 		\phi   \cdot \nu =0 &  \hbox{ on }\partial \Omega \setminus \Gamma_D  , 
 	\end{array}\right. 
 \end{equation} 
 where $F^*$ and $G^*$ denote  the usual Legendre transforms of $F$ and $G$ respectively (see Section \ref{Sassumptions} for the definitions).

To simplify the presentation,  we focus in this paper on the case where   $G (x,-\phi    \cdot \nu) $ is proportional to $-\phi    \cdot \nu.$ More precisely, we assume that    
$$G (x,r)  = g (x)\: r,\quad \hbox{ for any } r\in \RR \hbox{ and a.e. }x\in   \Gamma_{D  },$$
for a given $g $ assumed to be the trace of a Sobolev function.    In some sense, $g $ represents some given charge when the mass $\rho $ moves  onto the boundary $\Gamma_{D }.$  As we will see,  the  boundary condition on $\Gamma_D$ simply turns into non-homogeneous Dirichlet boundary condition in this case. The right assumptions on  $\Omega,$ $\Gamma_D,$ $F$, $g$   rigorous definition of $I,$ and its connection with \eqref{Eformal} typed PDE     will be given in details in the next section.  

This being said,  we can consider a dynamic system in $\Omega$ with an internal energy $\int_\Omega \beta(.,\rho )\: dx,$ associated with   densities $\rho  .$     The proximal optimization algorithm aims to construct a sequence $\rho^i$  initialized  with a given  initial density $\rho_{0}$.  For  $i=0,1,...n,$  the algorithm  establishes $\rho^i$,   by minimizing the  problem: 
\begin{equation}\label{STdesc1}
	 \hbox{argmin}_{\rho } \left\{  \int_\Omega \beta(.,\rho )\: dx   + I^\tau  (\rho^{i-1},\rho)  \right\},
\end{equation}   where $\tau>0$  is a given time step and,      $I^\tau $ corresponds to the $\tau-$scaled transfer fees of   
 $\rho ^{i-1}$ into $\rho $ given through $I  (\rho^{i-1},\rho).$   It is constructed up to the minimum of the amount 
 \begin{equation} 
 	\int_\Omega F(., \phi )\: dx - \int_{\Gamma_D}   g\: \phi \cdot \nu  \: d\sigma ,
 \end{equation} 
 among the vector valued fluxes given by the balance PDE 
 \begin{equation} \label{Div2}
 	-\nabla \cdot \phi  = \rho^{i-1}-\rho  \quad \text{ in } \ \Omega.
 \end{equation} 
 Here, on the Dirichlet boundary $\Gamma_D,$ we'll  keep the normal trace of $\phi $ free to vary. And, we fix it on the remaining part $\partial\Omega\setminus \Gamma_D.$ This leads to Neumann boundary condition (not necessarily homogeneous) for the PDE \eqref{Div2} (see more details in Section \ref{SProximal}).

  \begin{remark}
  	We observe that the term $I  (\rho^{i-1},\rho)$  depends only on the quantity  $\rho^{i-1}-\rho$.   Without any ambiguity, we will consider $I  (\rho^{i-1},\rho)$ definitely as  $I  (\rho^{i-1}-\rho).$   Furthermore, we will utilize the balance equation \eqref{div} under general considerations, including non-homogeneous boundary conditions and dual Sobolev source terms. That is, we will consider equation \eqref{div} in its most general form to accommodate a wider range of practical applications, including the case of non-homogeneous boundary conditions and the case of a forcing term of the type $\nabla \cdot \chi$ (see the exact definition of $I$ in \eqref{Imu} and Remark \ref{Rforcing}).   
  \end{remark}
    
  	\section{Notations and assumptions} \label{SAssumptions}
  	\setcounter{equation}{0}

  	Assume that $\Omega$ is bounded regular domain with Lipschitz boundary $\partial \Omega$ satisfying 
  	$$\partial \Omega=\Gamma_{D}\cup\Gamma_N ,$$  
  	where $\Gamma_{D}$  and $ \Gamma_N$ are assumed to be  disjoint regular connected components. We assume moreover that 
  	$$\mathcal H^{N-1}(\Gamma_D)>0.  $$

  	\subsection{Reminder}  
  	
  Let  $1<  p< \infty$ and  $1<  p'< \infty$ be such that 
  $$\frac{1}{p}+ \frac{1}{p'} =1.  $$
  We focus in this paper on the case where $1<p<\infty.$

  \begin{itemize}
  	
  		\item 
  	$L^{ p}(\Omega)$ and $L^1(\Omega)$ denotes the usual Lebesgue  spaces endowed with their  natural norms.   Unless otherwise stated, and without abuse of notation, we will use $\int f$  to denote the Lebesgue integral over $\Omega,$ i.e. 
  	$$\int_\Omega  h := \int_\Omega h(x)\: dx, \quad \hbox{ for any }h\in \L^1(\Omega). $$

  	\item 
 $W^{1,p}(\Omega)$ denotes the usual Sobolev space endowed with its natural norm 
  	$$ \Vert z\Vert_{W^{1,p}(\Omega)} =  \left( \int_\Omega  \vert z\vert^p +\int_\Omega  \vert \nabla z\vert^p \right) ^{1/p}.$$
  	\item  $W^{1,p}_{\Gamma_D} (\Omega) $ is  the closure, in $W^{1,p}(\Omega)$,  of $\mathcal C^1(\overline \Omega)$ which are null on $\Gamma_D.$   In particular, for any 
  	$z\in W^{1,p}_{\Gamma_D} (\Omega),$  $\gamma(z)=0$ on $\Gamma_D,$  where   $\gamma $  denotes the usual trace application defined from $W^{1,p}(\Omega)\to L^p(\partial\Omega).$  
  	
  	\item  	We denote by $ W^{-1,p'}_{\Gamma_D}(\Omega)$, the dual space  of $W^{1,p}_{\Gamma_D}(\Omega).$ Thanks to Lemma \ref{LdualW}, we know that,  for any $f\in   W^{-1,p'}_{\Gamma_D}(\Omega),$ there exists a couple  $(f_0, \overline f)\in L^{p'}(\Omega)\times L^{p'}(\Omega) ^N,$ such that  
  	\begin{equation}\label{decompositionH}
  		\langle f,\xi\rangle_{W^{-1,p'}_{\Gamma_D}(\Omega),W^{1,p}_{\Gamma_D}(\Omega)}  = \int_\Omega f_0\:\xi  -   \int_\Omega \overline f\cdot \nabla \xi  ,\quad \hbox{ for any }\xi\in 
  		W^{1,p}_{\Gamma_D}(\Omega). 
  	\end{equation}
  	Without abusing, we will use the couple $(f_0,\overline f)$ to identify    $f\in 	 W^{-1,p'}_{\Gamma_D}(\Omega)$ and denote the duality bracket by $\langle .,.\rangle_{\Omega} $ ; i.e.  
  	$$\langle f,\xi\rangle_{\Omega}:= \langle f,\xi\rangle_{W^{-1,p'}_{\Gamma_D}(\Omega),W^{1,p}_{\Gamma_D}(\Omega)},\quad \forall \xi\in W^{1,p}_{\Gamma_D}(\Omega).$$ 
  	
  	  	\item   For any  $\tilde\Gamma \subset \partial\Omega ,$ we denote by $W^{1-1/p,p}(\tilde\Gamma):= \gamma_{/\tilde\Gamma}(W^{1,p}(\Omega)),$ where $ \gamma_{/\tilde\Gamma}$ is the trace application  restricted to $\tilde\Gamma.$  We  need to define moreover   
  	$$  W^{1-1/p,p}_{00}(\tilde\Gamma):=\Big\{  \kappa\in L^{p'}(\tilde\Gamma)\: :\:  \exists \: \tilde \kappa\in W^{1,p} (\Omega),\: \gamma(\tilde \kappa)=\kappa \hbox{ on }  \tilde\Gamma  \hbox{ and } \tilde \kappa=0 \hbox{ on }\partial \Omega \setminus \tilde \Gamma  \Big\}.$$
Thanks to the assumptions on $\partial\Omega,$ taking $\tilde \Gamma=\Gamma_D$ or $\tilde \Gamma_N,$ the space  $W^{1-1/p,p}_{00}(\tilde\Gamma)$  coincides with the set of     functions belonging to   $W^{1-1/p,p}(\tilde\Gamma)$   and vanishing  on the   boundary  $\partial \Omega\setminus \tilde \Gamma.$  Remember that this not automatically true for any $\tilde\Gamma \subset \partial\Omega $.   

    \item   For any  $\tilde\Gamma \subset \partial\Omega , $ 
  we  denote by $ W^{-1/{p'},{p'}}(\tilde\Gamma):= \left( W^{1-1/p,{p}}_{00}( \tilde\Gamma) \right)'  $,  the topological dual space of   $W^{1-1/p,p}_{00} ( \tilde\Gamma). $  We denote the duality bracket   simply by the    formal expression $\langle \pi ,\kappa\rangle_{\tilde \Gamma}$ ; i.e. \begin{equation} \label{notation1}
  \langle \pi ,\kappa\rangle_{\tilde \Gamma} :=  \langle\pi,\kappa \rangle_{ W^{-1/p',p'}( \tilde\Gamma),  W^{1-1/p,p}_{00}( \tilde\Gamma)} ,\quad \forall \kappa \in W^{1-1/p,p}_{00}( \tilde\Gamma) \hbox{ and }\pi \in W^{-1/{p'},{p'}}(\tilde\Gamma)   .
  	\end{equation}

    	\item  	We can define simultaneously the space  
  	$$ H^{p'}(div,\Omega):=\Big\{  \upsilon\in L^{p'}(\Omega)^N \: :\:   \nabla \cdot \upsilon  \in L^{p'}(\Omega)
  	\Big\} ,$$
  	where $\nabla \cdot \upsilon$ is taken in $\mathcal D'(\Omega).$    
  	The space $ H^{p'}(div,\Omega)$ endowed with the norm  
  	$$	\Vert \upsilon\Vert_{ H^{p'}(div,\Omega)}  :=  \Vert \upsilon\Vert _{L^{p'}(\Omega)}+ \Vert \nabla \cdot \upsilon\Vert _{L^{p'}(\Omega)}     $$ 
  	is a Banach space.

  	  	\item  Let $\upsilon\in H^{p'}(div,\Omega).$  For any  $\tilde\Gamma \subset \partial\Omega ,$ the normal trace  $\upsilon\cdot \nu$ is well defined on $\tilde\Gamma$ by duality.   More precisely,   $\upsilon\cdot \nu\in W^{-1/{p'},{p'}}(\tilde\Gamma) $,    and  
  	\begin{equation}\label{tracexpression}
  		\langle \upsilon \cdot \nu ,\kappa  \rangle_{  \tilde\Gamma}  =\int_\Omega\upsilon \cdot \nabla \tilde \kappa  \: dx +  \int_\Omega \tilde \kappa \: \nabla \cdot \upsilon\: dx, 
  	\end{equation}  
  	for any $\kappa\in  W^{1-1/p,p}(\tilde\Gamma)$ and   $\tilde \kappa \in  W^{1,p} (\Omega)$  such that  $\tilde \kappa =\kappa$ on $\tilde\Gamma  , $    and $\tilde \kappa =0$ on  $ \partial\Omega\setminus \tilde\Gamma .$ 
  
  	  \end{itemize}

  	\medskip 
  	
  	We conclude this reminder section    with the following result, which is relatively well-known in the field of elliptic PDEs. 
  	
  	\begin{lemma}\label{Fnonempty}
     For any  $\mu\in L^{p'}(\Omega),$ $\chi\in L^{p'}(\Omega)^N$and $\eta\in  W^{-1/{p'},{p'}}(\Gamma_N),$ there exists  $\phi \in  H^{p'}(div,\Omega)$ solution of the following PDE :
  		\begin{equation}\label{diveq}
  			\left\{  \begin{array}{ll}
  				-\nabla \cdot  \phi =\mu +\nabla \cdot \chi \quad &\hbox{ in }\Omega\\  \\ 
  				(\phi +\chi) \cdot \nu =\pi &\hbox{ on }\Gamma_N,
  			\end{array}\right.
  		\end{equation}
  		in the sense that 
  		\begin{equation}\label{soldiv}
  			\int_\Omega (\phi +\chi) \cdot \nabla \xi \: dx = \int_{\Omega }\mu\: \xi\: dx +\langle \pi, \xi\rangle_{\Gamma_N}  , \quad \hbox{ for any }\xi\in W^{1,p}_{\Gamma_D } (\Omega).
  		\end{equation}
  	\end{lemma}
  	\begin{proof}
  		It is enough to work with  the optimization problem 
  		\begin{equation}\label{otimp}
  			\min_{z\in W^{1,p}_{ \Gamma_D}(\Omega)} \left\{ \frac{1}{p} \int_\Omega \vert \nabla z\vert^p\: dx -\int_\Omega \chi\cdot \nabla z\: dx +\int_\Omega \mu\: z\: dx +\langle \pi, z\rangle_{\Gamma_N}  \right\}.
  		\end{equation}
  		Using standard argument of calculus of variation, one can prove that   \eqref{otimp} has a solution that we denote by $u_p.$ Moreover, we can prove that $u_p$ satisfies  
  		\begin{equation} 
  			\int_\Omega \vert \nabla u_p\vert^{p-2}\nabla  u_p \cdot \nabla \xi \: dx = \int_{\Omega }\mu\: \xi\: dx - \int_\Omega \chi\cdot \nabla \xi \: dx + \int_{\Gamma_N } [\pi, \xi], \quad \hbox{ for any }\xi\in W^{1,p}_{\Gamma_D } (\Omega).
  		\end{equation} 
  		Taking $\phi = \vert \nabla u_p\vert^{p-2}\nabla  u_p,$ the proof of the lemma is complete. 
  	\end{proof}

  	\begin{remark}
  		In terms of  $W^{-1,p'}_{\Gamma_D}(\Omega),$   the equation \eqref{diveq} is equivalent to write 
  			$-\nabla \cdot  \phi  =f$ in  $W^{-1,p'}_{\Gamma_D}(\Omega),$ where $f$ is given by 
  			\begin{equation} 
  			\langle f,\xi\rangle_{\Omega}  = \int_\Omega \mu \:\xi\: dx -   \int_\Omega \chi \cdot \nabla \xi  \: dx  + \langle \pi,\xi\rangle_{\Gamma_N} ,\quad \forall \xi\in 
  			W^{1,p}_{\Gamma_D}(\Omega). 
  		\end{equation}
  	\end{remark}

  	\subsection{Assumptions} \label{Sassumptions}
  	
  We consider the applications  $F  \: :\:  \Omega\times \RR^N \to \RR $  and $\beta\: :\:  \Omega\times \RR  \to \RR $  to be  two 
  	 Caratheodory applications ;  i.e. the applications  $ (x,\xi)\in\Omega\times  \RR^N \to F  (x,\xi)  \in \RR  $ and $ (x,r)\in\Omega\times  \RR  \to \beta  (x,r)  \in \RR  $   are measurable with respect to $x\in \Omega$, and continuous with respect to $\xi\in \RR^N $  and $r\in \RR,$ respectively. We assume moreover that    
  	\begin{itemize} 
  		
  		\item [(H1)]     there exists two constants $C_1,\: C_2 >0,$    such that 
  		
  		$$C_1\: \Vert \xi\Vert^{p'} \leq F  (x,\xi) \leq C_2\: \Vert \xi\Vert ^{p'},\quad \hbox{ for any } \xi\in \RR^N,$$
  		
  		\item [(H2)]    the application   $ \xi\in \RR^N \to F  (x,\xi)  \in \RR$ is    lower semi-continuous  (l.s.c. for short) and convex, for a.e. $x\in \Omega,$  	  
  		  		
  		\item [(H3)]     for a.e. $x\in \Omega$  and  for a.e. $x\in \Omega$ ,  the application 
  		$$ r \in \RR \to \beta(x,r )   \hbox{ is l.s.c., convex and }\beta(x,0)=0, $$  	
  		 		
  		\item [(H4)]    there exists $C_3,\: M>0$ such that 
  		$$C_3\: (\vert r\vert -M)^{+p'}\leq \beta(x,r),\quad \hbox{ for any }r\in \RR \hbox{ and a.e. }x\in \Omega.$$   
  		
  	\end{itemize}    
  	
  	Under the  assumptions (H1)-(H4), we can define the Legendre transforms of $F$ and $\beta$  in the second variable as usual by 
  	$$F^*(x,\xi^*)= \max_{\xi\in \RR^N} \Big\{ \xi^*\cdot \xi -F(x,\xi)\Big\} ,\hbox{ for a.e. }x\in \Omega\hbox{ and }\xi^*\in \RR^N  $$
  	and 
  		$$\beta^*(x,r^*)= \max_{r\in \RR } \Big\{ r^*\cdot r -\beta(x,r)\Big\} ,\hbox{ for a.e. }x\in \Omega\hbox{ and }r^*\in \RR. $$
  	Let us remind the reader here that  $F^*$ and $\beta^*$ are again l.s.c. and convex with respect to the second variable. Moreover, 
  	$$F^*(x,\xi^*)=   \xi^*\cdot \xi -F(x,\xi)\quad \Leftrightarrow \quad \xi^*\in  \partial_\xi F(x,\xi)$$ and 
  $$\beta^*(x,r^*)\in  r^*\cdot r -\beta(x,r) \quad \Leftrightarrow \quad  r^* \in \partial_{r}\beta(x,r).$$  These are also  equivalent  to 
  $\xi\in  \partial_{\xi^*} F^*(x,\xi^*)$ and $r \in \partial_{r^*}\beta^*(x,r^*),$ respectively.   Moreover, thanks to the assumptions $(H1),$   it is not difficult to see that 	
  \begin{itemize}  
  	\item [(H'1)]     there exists two constants $C'_1,\: C'_2 >0,$    such that 
  \end{itemize}  
  $$C'_1\: \Vert \xi^* \Vert^{p} \leq F^*  (x,\xi^* ) \leq C'_2\: \Vert \xi^* \Vert ^{p},\quad \hbox{ for any } \xi^* \in \RR^N\hbox{ and  a.e. }x\in \Omega. $$ 
 For the main facts and classical results about conjugate convex functions, one can see for instance the book \cite{Ekeland}.

  		\begin{remark}
  	  The  results of the paper remains to be true if we replace the assumption   $(H3)$ by  the following one 
  			$$C_1\: \Vert \xi\Vert^{p'} - c_1(x)\leq F  (x,\xi) \leq C_2\: \Vert \xi\Vert ^{p'} + c_2(x),\quad \hbox{ for any } \xi\in \RR^N, \hbox{ and  a.e. }x\in \Omega,$$
  			where  $c_1$ and $c_2$ are two $L^1$ functions on $\Omega.$   
  			Also, they remains to be true   if we replace the assumption $(H4)$ by 
  			$$C_3\: (\vert r\vert -M)^{+p'}- c_3(x)\leq \beta(r),\quad \hbox{ for any }r\in \RR \hbox{ and  a.e. }x\in \Omega,$$    
  			where  $c_3$ is a  $L^1$ function on $\Omega.$
  						
  		  \end{remark}

  	\section{Proximal  energy minimization}\label{SProximal}
  	\setcounter{equation}{0}

  	\subsection{Main results}
  	Let   $1<p<\infty$ 
  	\begin{equation} \label{hypggamma}
  		\pi  \in  W^{-\frac{1}{p'},p'}(\Gamma_N) \quad \hbox{ and }\quad g \in W ^{1-\frac{1}{p},p}_{00}(\Gamma_D)   \end{equation} 
  	be given and fixed throughout this paper.    We denote by $\tilde g $ the  $W^{1,p} (\Omega)$ function such that   $\tilde g =0$ on $\Gamma_N$  and    $\tilde g =g $ on $\Gamma_D .$   To define the transition work   in \eqref{STdesc1} rigorously, 	 for any $\mu \in L^{p'}(\Omega)$  and $\chi   \in  \left(L^{p'}(\Omega)\right)^N$,  we consider   
  	$$ \A_{\pi}^\chi(\mu)= \Big\{  \phi   \in    L^{p'}(\Omega)^N,\: - \nabla \cdot \phi  =\mu +\nabla \cdot \chi  \hbox{ in }\Omega \hbox{ and }   (\phi  +\chi)\cdot \nu =\pi  \hbox{ on }\Gamma_{N }    \Big\}.  $$
  	That is $  \phi     \in \A_{\pi}^\chi(\mu,h)$ if and only if $\phi  \in   L^{p'}(\Omega)^N$ and 
  	$$\int_\Omega\phi \cdot \nabla \xi \: dx =\int_\Omega \mu\:  \xi  \: dx -\int_\Omega \chi\cdot \nabla \xi\: dx +\langle \pi,\xi\rangle_{\Gamma_N}\quad \hbox{ for any }\xi\in W^{1,p}_{\Gamma_{D  }}(\Omega). $$
  	  Thanks to Lemma \ref{Fnonempty}, under the assumptions \eqref{hypggamma},  we have 
		$$  \A_{\pi}^\chi(\mu)  \neq\emptyset,\quad \hbox{ for any }\mu\in  L^{p'}(\Omega)  \hbox{ and }\chi\in   L^{p'}(\Omega)^N  . $$  
 Now, for any  $\mu \in  L^{p'}(\Omega) $ and $\chi  \in   L^{p'}(\Omega)^N  ,$  fixed,  we  define transition work as follows 
  \begin{equation}\label{Imu}
  	\I_{\pi,g}^\chi( \mu) := 	\inf_{\phi}\left \{   \int_\Omega F  (x,  \phi  (x) )\:  d x -\langle (\phi+\chi)\cdot \nu, g\rangle_{\Gamma_D} \:  :\:    \phi\in  \A_{\pi}^\chi(\mu)     \right\}.
  \end{equation}

  	\begin{remark}  \label{Rforcing}
  		 As we said in Section 2.2, the quantity  $ \I_{\pi,g}^\chi( \mu) $ may be considered as  the transition work associated  with the distribution of mass  $\mu$. However, revisiting the formulation \eqref{div} we note that in the definitions of $\A$ and $\I$, the source term of \eqref{div} is generalized to a more comprehensive form, $\mu +\nabla \cdot \chi .$
  		  Indeed,   the term $\nabla \cdot \chi $  can be utilized to represent a significant external influence exerted on the system by an external force, $\chi$ (along with boundary conditions).  In some sense   $ \I_{\pi,g}^\chi( \mu) $ may be  interpreted as a \emph{transition work balancing the distribution  $\mu$ under   external force $\chi$}. 
  		 The main future of this interpretation   is to encompass a broad spectrum of applications, including PDEs of the type \eqref{E0} with a fixed forcing term incorporated into the source term, $f.$ Moreover, it can handle also PDEs of the type \eqref{EDT}, in which the forcing term is represented by a transport phenomenon clearly identified in the first-order term, $\nabla \cdot (\rho\: V).$


  \end{remark}

 Then, to solve the steepest descent algorithm associated with transition work $\I_{\pi,g}^\chi$  and the  internal energy  $\int_\Omega   \beta(.,\rho),$    we consider the    proximal  energy minimization problem 
   \begin{equation}\label{Nmu}
   	 \N_{\pi,g}^\chi (\mu)   :=    	 	\inf_{\rho}\left \{\int_\Omega   \beta(.,\rho) +  \I_{\pi,g}^\chi( \mu-\rho) ,\: \rho\in  L^{p'}(\Omega) \right\}      ; 
   \end{equation}
i.e. 
  $$  \N_{\pi,g}^\chi (\mu)   :=    	 	\inf_{\rho,\phi }\left \{\int_\Omega   \beta(.,\rho) +  \int_\Omega F(.,\phi ) \: dx -  \langle (\phi +\chi)\cdot  \nu, g\rangle_{\Gamma_D}  \: :\:     \phi  \in  \A_{\pi}^\chi(\mu-\rho) ,\: \rho\in  L^{p'}(\Omega) \right\}.  $$    
  See that  all the terms in $   \N_{\pi,g}^\chi (\mu)  $ are well defined. Indeed, for any $  \phi  \in  \A_{\pi}^\chi(\mu-\rho) ,$ we have $\phi  +\chi \in H^{p'}(div, \Omega),$     so that   the normal trace of $\phi  +\chi $ is well defined on $\Gamma .$   Remember that  $ \langle (\phi +\chi)\cdot  \nu, g\rangle_{\Gamma_D}  $    needs  to be understood  in the sense of 
  \begin{eqnarray*}
  \langle (\phi +\chi)\cdot  \nu, g\rangle_{\Gamma_D}  
  	&=&  \int_\Omega(\phi +\chi)  \cdot \nabla \tilde g   \: dx +  \int_\Omega \tilde g  \: \nabla \cdot (\phi +\chi) \: dx .   
  \end{eqnarray*}

  \begin{remark}
  	
  	The optimization problem  $ \N_{\pi,g}^\chi( \mu) $  seeks to identify the proximal state, $\rho$ for    $\mu,$  relative to the total energy,  which simultaneously minimizes the internal energy $\int_\Omega \beta(x,\rho)\: dx$  and the transition work $ \I_{\pi,g}^\chi$ required to reach a proximal state of $\mu$  driven by  the forcing term $\chi$. Notably, no assumption is made regarding the signs of either $\mu$ or $\rho$.  Nevertheless, practical considerations may necessitate imposing sign constraints on $\rho.$  This  must be carefully considered in subsequent analysis and interpretation in terms of   PDE (see Remark \ref{Rsigne}).

  \end{remark}

  \bigskip 
 Our first main result of this section is the following characterization in terms of PDE of the solution to \eqref{Imu}.

 \begin{theorem}	\label{tduality1}
 	For any $\mu \in  L^{p'}(\Omega) $ and $\chi \in   L^{p'}(\Omega)^N  ,$   we have  :  
 	\begin{equation}\label{duality1}
 		\I_{\pi,g}^\chi( \mu)  = 	\underbrace{ \max_{\eta\in W ^{1,p}(\Omega) ,\: \eta_{/\Gamma_D}=g} \left\{  \int_\Omega  \mu\:\eta-\int_\Omega \chi\cdot \nabla\eta  - \int_\Omega F^*(.,\nabla\eta)    + \langle \pi,\eta\rangle_{\Gamma_N}  \right\}}_{ =:\J_{\pi,g}^\chi( \mu) } . 
 	\end{equation}  
 	Moreover, \begin{enumerate}
 		\item 	$\I_{\pi,g}^\chi( \mu) $ has a solution $\phi  \in  L^{p'}(\Omega) ^N.$ 
 		\item $\phi $ and $\eta$   are solutions of  $   \I_{\pi,g}^\chi( \mu)  $ and $ \J_{\pi,g}^\chi( \mu)  ,$ respectively,  if and only if     they satisfies the following   PDE :  
 	\begin{equation}
 		\label{PDE1}
 		\left\{\begin{array}{ll} 
 			\left.  \begin{array}{ll} 
 				 - \: \nabla\cdot \phi    =  \mu   +\nabla \cdot \chi   \\ 
 				\\  		 \phi    =  \partial_\xi F  ^*(., \nabla \eta )    \\   \\ 	 
 			\end{array}\right\} 	\quad & \hbox{ in } \Omega, \\  
 			(\phi   +\chi )  \cdot \nu  =  \pi    & \hbox{ on }\Gamma_{N  }\\  \\
 			\eta   =  g   \quad    & \hbox{ on }\Gamma_{D  }, 
 		\end{array}
 		\right. 	 
 	\end{equation} 
 	in the sense that   $\eta \in W^{1,p} (\Omega),$  $\eta =g $ on $\Gamma_{D  },$    	$\phi   \in L^{p'}(\Omega)^N,$  $\phi   \in \partial   F^*  (.,\nabla \eta )$ a.e. in $\Omega$,  and we have  $$  \int_\Omega \phi   \cdot \nabla \xi  \: dx=\int_\Omega \mu  \: \xi \: dx  -\int_\Omega \chi\cdot \nabla \xi  \: dx +\langle \pi , \xi \rangle_{\Gamma_N}   ,\quad \forall \xi\in W ^{1,p}_{\Gamma_{D} } (\Omega)   .$$   
   \end{enumerate} 
 \end{theorem}

 \bigskip 
As to  \eqref{Nmu}, we have 
  
  \begin{theorem}	\label{tduality2}
  	For any $\mu \in  L^{p'}(\Omega) $ and $\chi \in   L^{p'}(\Omega)^N  ,$   we have  :  
  	\begin{equation}
  		\N_{\pi,g}^\chi (\mu)  = 	\underbrace{ \max_{\eta\in W ^{1,p}(\Omega) ,\: \eta_{/\Gamma_D}=g} \left\{  \int_\Omega  \mu\:\eta-\int_\Omega \chi\cdot \nabla\eta  - \int_\Omega F^*(.,\nabla\eta)   -\int_{\Omega } \beta^*(.,\eta)   + \langle \pi,\eta\rangle_{\Gamma_N}  \right\}}_{ =:D_{\pi,g}^\chi( \mu) } . 
  	\end{equation}  
  	Moreover,  \begin{enumerate}
  	 		\item 	$\N_{\pi,g}^\chi (\mu)  $ has a solution $(\rho,\phi )   \in  L^{p'}(\Omega) \times  L^{p'}(\Omega)^N.$  
  	 		\item 	$(\rho,\phi )$ and $\eta$ are solutions of  $   \N_{\pi,g}^\chi (\mu)  $ and $ D_{\pi,g}^\chi( \mu)  ,$ respectively,  if and only if     they satisfies the following   PDE :  
  	\begin{equation}
  		\label{PDE2}
  		\left\{\begin{array}{ll} 
  			\left.  \begin{array}{ll} 
  				\rho  - \: \nabla\cdot \phi    =  \mu   +\nabla \cdot \chi  ,    \\ 
  				\\  				\rho \in \partial \beta^*(x,\eta  ),\:   \phi    =  \partial_\xi F  ^*(., \nabla \eta )    \\   \\ 	 
  			\end{array}\right\} 	\quad & \hbox{ in } \Omega, \\ 
  			(\phi   +\chi )  \cdot \nu  =  \pi    & \hbox{ on }\Gamma_{N  }\\  \\
  			\eta   =  g   \quad    & \hbox{ on }\Gamma_{D  }
  		\end{array}
  		\right. 	 
  	\end{equation} 
  	in the sense that $\rho  \in L^{p'}(\Omega),$  $\eta \in W^{1,p} (\Omega),$    $\phi   \in L^{p'}(\Omega)^N,$  $\eta =g $ on $\Gamma_{D  },$ $\phi   \in \partial  F^*  (.,\nabla \eta ),$       $\rho  \in \partial\beta^* (x, \eta) ,$   a.e. in $\Omega,$ and  
    		 $$\int_\Omega  \rho  \: \xi   + \int_\Omega( \phi + \chi)   \cdot \nabla \xi  =\int_\Omega \mu  \: \xi    +\langle \pi , \xi \rangle_{\Gamma_N} ,\quad \forall \xi\in W ^{1,p}_{\Gamma_{D} } (\Omega)   .$$
  	   \end{enumerate} 
  \end{theorem}

  \bigskip

  	  \begin{remark}\label{Rsigne}
  \begin{enumerate}
  	\item 
  We do not treat   uniqueness concerns  of the solution for the problem  	\eqref{PDE2}. However, one sees that primal   issues in this direction can be deduced by using the optimization problems $\N_{\pi,g}^\chi,$ $\J_{\pi,g}^\chi$  and strict convexity of  $\beta$, $F,$ $\beta^*$ and/or $F^*.$ 
  	 
  	   \item  	In certain practical situations, the requirement that the density $\rho$  in the proximal energy optimization problem  $\N_{\pi,g}^\chi$ remains non-negative arises. The feasibility of this constraint hinges on the specific conditions imposed on the data $\mu,$ $\chi,$ $g$ and $\pi$.  	For instance, if $g\equiv 0,$ $\mu+\nabla \cdot \chi \geq 0$ and $\pi\geq 0,$  it is possible to prove that the solution $\rho$  to the problem $\N_{\pi,0} (\mu)  $ is conceivable non-negative, by utilizing the associated PDE \eqref{PDE2}.
  	 In the absence of these conditions, one must explicitly incorporate the constraint $\rho\geq 0$  into the optimization problem $\N  $, thereby necessitating the introduction of Lagrangian multipliers.
  	 
  	   \end{enumerate}	 
  	 \end{remark} 
  	
  The main future of the PDE \eqref{PDE1} deserves the  evolution dynamic of the type 
  \begin{equation} 
  	\left\{\begin{array}{ll} 
  		\left.  \begin{array}{ll} 
  			\partial_t\rho  - \nabla\cdot (\phi  - V(x,\rho,\eta))  =  \mu   ,\quad   \\ 
  			\\  				\rho \in \partial \beta^*(x,\eta  ),\quad  \phi    =  \partial_\xi F  ^*(., \nabla \eta )  \\   \\ 	 
  		\end{array}\right\} 	\quad & \hbox{ in } \Omega, \\  \\
  		(\phi   -V )  \cdot \nu  =  \pi    & \hbox{ on }\Gamma_{N  }\\  \\
  		\eta   =  g   \quad    & \hbox{ on }\Gamma_{D  }\\  \\ 
  		\rho(0)=\rho_0
  	\end{array}
  	\right. 	 
  \end{equation}  
  In Section 5, we show how to use this analysis to solve this kind of problems     in the case where 
  $$V(x,\rho,\eta) =\rho\: V,$$ 
  where $V$ is a given vector field.

  \subsection{Proofs of Theorem \ref{tduality1} and Theorem \ref{tduality2}}\label{Sproofstat}

  The proofs of   Theorem \ref{tduality1} and Theorem \ref{tduality2} are connected. To simplify the presentation of both proofs, we introduce  the problems 
  
   $$  \tilde \N_\alpha (\mu)   :=    	 	\inf_{\rho,\phi }\left \{\alpha \int_\Omega   \beta(.,\rho) +  \int_\Omega F(.,\phi ) \: dx -  \langle (\phi +\chi)\cdot  \nu, g\rangle_{\Gamma_D}  \: :\:     \phi  \in  \A_{\pi}^\chi(\mu-\alpha \:  \rho) ,\: \rho\in  L^{p'}(\Omega)  \right\}.  $$    
 and 
$$  \tilde D_\alpha   :=    	  
  \max_{\eta\in W ^{1,p}(\Omega) ,\: \eta_{/\Gamma_D}=g} \left\{  \int_\Omega  \mu\:\eta-\int_\Omega \chi\cdot \nabla\eta  - \int_\Omega F^*(.,\nabla\eta) - \alpha \int_\Omega   \beta^*(.,\eta)   + \langle \pi,\eta\rangle_{\Gamma_N}  \right\},$$  
  where $\alpha \in \{0,1\}.$ 
Then, it is clear that $\I_{\pi,g}^\chi(\mu) = \tilde \N_0  (\mu) $,  $\N_{\pi,g}^\chi(\mu) = \tilde \N_1(\mu) ,$  $\J_{\pi,g}^\chi(\mu) = \tilde D_0  (\mu) $ and   $D_{\pi,g}^\chi(\mu) = \tilde D_1(\mu) ,$  So, it is    enough to prove that  $\tilde \N_\alpha(\mu) = \tilde D_\alpha (\mu) $ ,  for $\alpha \in \{0,1\}.$   
  
  \bigskip 
  To begin with, we prove the following lemma. 
  \begin{lemma}\label{Pexistprimal}
  	For any $\mu \in  L^{p'}(\Omega) ,$   $\chi\in   L^{p'}(\Omega)^N   $  and  $\alpha \in \{0,1\},$  the problems $\tilde \N_\alpha(\mu) $ and $\tilde D_\alpha (\mu) $   have  solutions. 
  \end{lemma}
  \begin{proof}
  	Let $(\rho_n^\alpha,\phi_n^\alpha)$ be a minimizing sequence for $   D (\mu)  .$ Here, we take  $\rho_n^\alpha$  equal to $0$ in the case where  $\alpha =0.$ Thanks to the assumptions on $F$ and $\alpha\beta$,  one sees that $\rho_n $ and    $\phi_n$  are bounded in $ L^{p'}(\Omega)  $   and $  L^{p'}(\Omega)^N$,   respectively. So, there exists a sub-sequence that we denote again by  $(\rho_n^\alpha,\phi_n^\alpha )\in     L^{p'}(\Omega)\times  L^{p'}(\Omega)^N $, and $(\rho^\alpha,\phi^\alpha)\in   L^{p'}(\Omega)\times L^{p'}(\Omega)^N,$  such that 
  	$$\rho_n^\alpha\to \rho^\alpha,\quad \hbox{ in } L^{p'}(\Omega) \hbox{-weak}^*$$ 
  	and   $$\phi_n^\alpha\to \phi^\alpha,\quad \hbox{ in }  L^{p'}(\Omega)^N \hbox{-weak} ,$$
  	where $\rho^0\equiv 0.$ 
  	In addition, combining  \eqref{tracexpression} with the fact that   $\nabla \cdot \phi _n^\alpha=\mu+\nabla \cdot \chi -\rho_n^\alpha $ in $\Omega$ and $(\phi _n^\alpha+\chi)\cdot \nu=\pi$ on $\Gamma_N,$ we see that 
  	$$ \langle (\phi_n^\alpha+\chi)\cdot \nu,g\rangle_{\Gamma_D}\to  \langle (\phi^\alpha+\chi)\cdot \nu ,g\rangle_{\Gamma_D}. $$  
  	Clearly $(\rho^\alpha,\phi^\alpha)$ is an admissible test function for the optimization problem $   \tilde \N_\alpha  (\mu)  $ ; i.e. $(\rho^\alpha,\phi^\alpha )\in  \A_{\pi}^\chi(\mu-\alpha\rho^\alpha) $.  Then, using the l.s.c. and convexity of $ \beta $ and $F  ,$        we deduce that  $(\rho^\alpha,\phi ^\alpha )$ is a solution of the the optimization problem  $    \tilde \N_\alpha(\mu)  .$   The proof for 
  	$\tilde  D_\alpha (\mu)  $ follows more or less the same arguments.  To avoid redundancy, we will omit the details of the proof. \end{proof}
   
  \bigskip 
Now,  for any $\alpha \in \{0,1\},$ let us consider the application $K_\alpha \: :\:   W^{ -1,p'}_{\Gamma_D} (\Omega)  \to \RR$ given by 
   $$ K_\alpha (f)  :=    	 \int_\Omega f_0\:  \tilde g -\int_\Omega \overline f \cdot \nabla \tilde g + 	\inf_{\rho,\phi }\left \{ \alpha \int_\Omega   \beta(.,\rho) +  \int_\Omega F(.,\phi ) \: dx -  \langle (\phi +\chi+ \overline f )\cdot  \nu, g\rangle_{\Gamma_D}  \:\right.$$  $$\left. :\:     \phi  \in  \A_{\pi}^{\chi+\overline f}(\mu+f_0-\alpha \:  \rho),\: \rho\in  L^{p'}(\Omega)  \right\},  $$
   where   $f_0$ and $\overline f$ are given by the decomposition  $f=f_0+\nabla \cdot \overline f$ in $W^{-1,p'}_{\Gamma_D}(\Omega).$  
   Then, it  is clear that 
   $$\N_{\pi,g}^\chi (\mu) =  K_1 (0)  \quad \hbox{ and }\quad \I_{\pi,g}^\chi( \mu) =K_0 (0).$$ 
Thank to    \eqref{tracexpression},  we see  that  
$$ \langle (\phi  +\chi+ \overline f )\cdot  \nu, g\rangle_{\Gamma_D}  
= \int_\Omega (\phi +\chi+ \overline f)  \cdot \nabla \tilde g -  \int_\Omega  (f_{0}+\mu-\alpha \rho) \: \tilde g     ,$$  
for any  $ \rho\in  L^{p'}(\Omega)$ and $  \phi  \in  \A_{\pi}^{\chi+\overline f}(\mu+f_0-\alpha \:  \rho).$ This implies   that  $K_\alpha$ may be written as 
\begin{equation}\label{mofK}
\begin{array}{c} 
		 K   _\alpha (f)   =    	 	\inf_{\rho,\phi }\left \{ \alpha \int_\Omega   \beta(.,\rho) +  \int_\Omega F(.,\phi ) \: dx   - \int_\Omega (\phi  +\chi )  \cdot \nabla \tilde g  \: dx + \int_\Omega  (\mu-\alpha \rho) \: \tilde g  \: \right.  \\  \\  \: 
  \left. :\:     \phi  \in  \A_{\pi}^{\chi+\overline f}(\mu+f_0-\alpha \:  \rho),\: \rho\in  L^{p'}(\Omega)  \right\}. 
  \end{array}  \end{equation} 
  
The following duality result is the key to proving Theorem \ref{tduality1} and Theorem \ref{tduality2}. After,   we need  simply to compute the Legendre transform of  $K_\alpha$  to complete the proofs.

\begin{lemma}\label{Kconvecl.s.c.}
	For any $\alpha \in \{0,1\},$ we have 
	\begin{equation} \label{perturbduality}
		-K_\alpha (0) = -K_\alpha ^{**}(0) = \min_{ \eta \in W^{1,p}(\Omega)} K_\alpha ^{*}(\eta) ,
	\end{equation} 
\end{lemma}
\begin{proof}To prove  \eqref{perturbduality}, it is enough to prove that  	$K_\alpha $ is convex and l.s.c.,  for any $\alpha \in \{0,1\},$  and      conclude   by classical duality results  (cf.  \cite{Ekeland}).  \\	
	 \underline{Convexity :} For any $f,\: h\in W^{-1,p'} _{\Gamma_D}(\Omega),$ taking $(\rho_f,\phi _f)$ and $(\rho_h,\phi _h)$ the solutions corresponding to the optimization problems $K_\alpha (f)  $ and 	$K_\alpha (h)  $ respectively, one sees that  $$\underbrace{ t\phi _h+(1-t)\phi _f}_{\phi _t}\in \A_{\pi}^{\overline h_{t}}(\mu-\alpha\underbrace{(t\rho_h+(1-t)\rho_f)}_{\rho_t}+\underbrace{(th_0+(1-t)f_0)}_{h_{0t}} ),$$ where 
	 $\overline h_{t}= t\overline h+(1-t)\overline f,$ for any $t\in [0,1].$   Moreover, using the convexity of $\beta (x,.)$ and $F(x,.),$ we have 
$$\alpha\int\beta(.,\rho_t)	+ \:  \int_\Omega F(.,\phi _t) -	   \langle (\phi _t+\chi+\overline f_t)\cdot \nu,g\rangle_{\Gamma_D}     \leq t \left(  \alpha\int_\Omega \beta(.,\rho_h)	+ \:  \int_\Omega F(.,\phi _h) -	   \langle (\phi _h+\chi+\overline h) \cdot \nu,g\rangle_{\Gamma_D}  \right)$$ 
$$ +  (1-t)\left(  \alpha\int_\Omega\beta(.,\rho_f)	+ \:  \int_\Omega F(.,\phi _f) -	   \langle (\phi _f+\chi+\overline f)\cdot \nu,g\rangle_{\Gamma_D} \right),$$ 
 which implies that 
$$ \int_\Omega (th_0+(1-t)f_0)\:  \tilde g -\int_\Omega   (t \overline h +(1-t)\overline f)  \cdot \nabla \tilde g+ \alpha\int_\Omega\beta(.,\rho_t)	+ \:  \int_\Omega F(.,\phi _t) -	   \langle (\phi _t+\chi+\overline f_t)\cdot \nu,g\rangle_{\Gamma_D}   $$  $$  \leq  t K_\alpha (h)+(1-t)K_\alpha (f) ,$$
and 
$$    K_\alpha (th+(1-t)f)  \leq  t K_\alpha (h)+(1-t)K_\alpha (f)  . $$
\underline{Lower semi-continuity :}  Let us consider $f_n$ a sequence of $W^{-1,p'} _{\Gamma_D}(\Omega)$ which converges to $f.$ That is a sequence of $L^{p'}(\Omega)$ and $L^{p'}(\Omega)^N$ functions $f_{0n}$ and $\overline f_n$ respectively, such that 
$$ \int_\Omega f_{0n}\:\xi -\int_\Omega \overline f_n\cdot \nabla \xi \to   \int_\Omega f_{0}\:\xi -\int_\Omega \overline f\cdot \nabla \xi , \quad \forall \xi\in W^{1,p}_{\Gamma_D}(\Omega).$$
Let us prove that $K_\alpha (f) \leq \liminf_{n\to\infty } K_\alpha (f_n).$    Then,  let us consider  $(\rho_n,\phi _n)$ be  the solution corresponding to $K_\alpha (f_n)$ ; i.e.  
\begin{equation}\label{testphin}
	 \phi _n \in  \A_{\pi}^{\chi+\overline f_n}(\mu+f_{0n}-\alpha \:  \rho_n) .  
\end{equation}
and
$$	 K   _\alpha (f_n)   =    	  \alpha \int_\Omega   \beta(.,\rho_n) +  \int_\Omega F(.,\phi _n) \: dx   -\int_\Omega (\phi _n+\chi )  \cdot \nabla \tilde g  \: dx  + \int_\Omega  (\mu-\alpha \rho_n) \: \tilde g, $$
where we use \eqref{mofK}.    We can assume that $  K_\alpha (f_n) $ is bounded. So,  there exists $C<\infty $ such that   
\begin{equation} 
  \alpha \int_\Omega   \beta(.,\rho_n) +  \int_\Omega F(.,\phi _n) \: dx  \leq C    +   \int_\Omega (\phi _n+\chi )  \cdot \nabla \tilde g  \: dx  -  \int_\Omega  (\mu-\alpha \rho_n) \: \tilde g  
\end{equation} 
Using  assumptions  $(H1)$ and $(H4)$ with Young formula,  we see that   $  \rho_n $ and    $\phi_n$  are bounded in $ L^{p'}(\Omega)  $   and $  L^{p'}(\Omega)^N$,   respectively. So, there exists $(\rho,\phi)\in   L^{p'}(\Omega)\times L^{p'}(\Omega)^N $  and a sub-sequence that we denote again by  $(\rho_n,\phi_n )  $,  such that 
$$\rho_n\to \rho,\quad \hbox{ in } L^{p'}(\Omega) \hbox{-weak}^*$$ 
and   $$\phi_n\to \phi,\quad \hbox{ in }  L^{p'}(\Omega)^N \hbox{-weak} .$$ 
Using \eqref{testphin}, we have 
\begin{equation}
	\int_\Omega (\phi_n-\chi)\cdot \nabla \xi = \int_\Omega (\mu-\rho_n)\: \xi +\langle f_n,\xi\rangle_{\Omega} +\langle \pi,\xi\rangle_{\Gamma_N},\quad \forall \xi\in W^{1,p}_{\Gamma_D},   
\end{equation}
and, by  letting $n\to \infty,$ we get 
\begin{equation}
	\int_\Omega (\phi-\chi)\cdot \nabla \xi = \int_\Omega (\mu-\rho )\: \xi +\langle f,\xi\rangle_{\Omega}+\langle \pi,\xi\rangle_{\Gamma_N},  \quad \forall \xi\in W^{1,p}_{\Gamma_D}. 
\end{equation} 
This implies that  $  \phi  \in  \A_{\pi}^{\chi+\overline f}(\mu+f_0-\alpha \:  \rho) .$      Then, using the l.s.c. and convexity of $\beta $ and $F  ,$   we have 
\begin{eqnarray*}
 K_\alpha(f)&\leq &  \alpha \int_\Omega   \beta(.,\rho) +  \int_\Omega F(.,\phi ) \: dx   -\int_\Omega (\phi +\chi )  \cdot \nabla \tilde g  \: dx  + \int_\Omega  (\mu-\alpha \rho) \: \tilde g   \\  \\ 
 & \leq &    \liminf_{n\to\infty } \left\{ \alpha \int_\Omega   \beta(.,\rho_n) +  \int_\Omega F(.,\phi _n) \: dx   -\int_\Omega (\phi _n+\chi )  \cdot \nabla \tilde g  \: dx  +  \int_\Omega  (\mu-\alpha \rho_n) \: \tilde g \right\}  \\  \\  &=&      \liminf_{n\to\infty }  K_\alpha (f_n).\end{eqnarray*} 
Thus the result.

\end{proof}

 \begin{proofth}{Proof of Theorem \ref{tduality1}}   As we said above, it is enough to compute 
 	$K_0^{*}.$   By definition of $K_0^*,$ for any $\eta \in W^{1,p}_{\Gamma_D} (\Omega),$  we have 
 	$$ K_0^*(\eta) =\max\left\{  \langle f,\eta \rangle_{\Omega} -  K_0(f ) \: :\:   f\in W^{-1,p'} _{\Gamma_D}(\Omega)\right\}    $$
 	$$ = \max_{f_0,\overline f,\phi }\Big\{  \underbrace{\int_\Omega  f_0   \:  \eta  -\int _\Omega \overline f \cdot \nabla  \eta 	- \:  \int_\Omega F(.,\phi ) +\int_\Omega (\phi  +\chi )  \cdot \nabla \tilde g  \: dx  -\int_\Omega  \mu \: \tilde g  }_{H(f_0,\overline f,\phi )}  $$ 
 	$$    \: :\:  (f_0,\overline f)\in L^{p'}(\Omega)\times L^{p'}(\Omega)^N,\:  \phi \in \A_{\pi}^{\chi+\overline f} (\mu+f_0)\Big\}.    $$
 	Since $\phi \in \A_{\pi}^{\chi+\overline f} (\mu+f_0) $ and $\eta \in W^{1,p}_{\Gamma_D} (\Omega),$ we have 
  $$ \int_\Omega  f_0   \:  \eta  -\int _\Omega \overline f \cdot \nabla  \eta  =  	\int_\Omega  (\phi+\chi)  \cdot \nabla \eta  - \int_\Omega   \mu\: \eta    -\langle \pi,\eta\rangle_{\Gamma_N}  .$$ 
 	This  implies that 
 	\begin{eqnarray*}
 		H(f_0,\overline f,\phi )  	&=& \int_\Omega  (\phi \cdot \nabla (\eta+\tilde g)-F(.,\phi ) )  - \int_\Omega  \mu\: (\eta+\tilde g)   + \int_\Omega  \chi\cdot \nabla (\eta+\tilde g)  -\langle \pi,\eta\rangle_{\Gamma_N} .\end{eqnarray*}
for any $ (f_0,\overline f)\in L^{p'}(\Omega)\times L^{p'}(\Omega)^N$ and $ \phi \in \A_{\pi}^{\chi+\overline f} (\mu+f_0) 
,$  	so that    \begin{eqnarray*}
 		K_0^*(\eta)
 		&=& 
 		\max_{ \phi \in  L^{p'}(\Omega)^N }\left\{  \int_\Omega (\phi \cdot \nabla (\eta+\tilde g)-F(.,\phi ) )   - \int_\Omega \mu\: (\eta+\tilde g)    + \int_\Omega  \chi\cdot \nabla (\eta+\tilde g)   -\langle \pi,\eta\rangle_{\Gamma_N}      \right\}  .
 	\end{eqnarray*}
Under  the assumptions $(H1)$,    we deduce   
 	\begin{eqnarray*}
 		K_0^*(\eta)&= &\int_\Omega F^*(.,\nabla (\eta+\tilde g)) - \int_\Omega \mu\: (\eta+\tilde g)   + \int_\Omega \chi\cdot \nabla (\eta+\tilde g)   -\langle \pi,\eta\rangle_{\Gamma_N} .
 	\end{eqnarray*}
 Then by  using \eqref{perturbduality}, we conclude that  
 	\begin{eqnarray*}
 		-K_0(0)&=&  
 		\min_{ \eta \in W^{1,p}_{\Gamma_D}(\Omega)}  \int_\Omega F^*(.,\nabla (\eta+\tilde g))   - \int_\Omega \mu\: (\eta+\tilde g)   + \int_\Omega \chi\cdot \nabla (\eta+\tilde g)   -\langle \pi,\eta\rangle_{\Gamma_N} \\  \\
 		&=& \min_{ \eta \in W^{1,p}(\Omega),\: \eta_{/\Gamma_D}=g}  \int_\Omega F^*(.,\nabla \eta) - \int_\Omega \mu\: \eta    + \int_\Omega  \chi\cdot \nabla \eta    -\langle \pi,\eta\rangle_{\Gamma_N}. 
 	\end{eqnarray*}
 	Thus the duality 
 	$$\I_{\pi,g}^\chi( \mu) = \J_{\pi,g}^\chi( \mu) . $$
 Thanks to Lemma \ref{Pexistprimal}, let us consider $\phi $ and $\eta$ the solutions of $\I_{\pi,g}^\chi( \mu) $ and $\J_{\pi,g}^\chi( \mu) .$  Our aim now is to prove that the couple $(\eta,\phi )$ solve the PDE \eqref{PDE1}.  The duality \eqref{duality1}  implies 
 		$$   \int_\Omega   \mu\:\eta-\int_\Omega \chi\cdot \nabla\eta  - \int_\Omega F^*(.,\nabla\eta)    + \langle \pi,\eta\rangle_{\Gamma_N}   =    \int_\Omega F(.,\phi ) \: dx -  \langle (\phi +\chi)\cdot  \nu, g\rangle_{\Gamma_D}  .$$  Since $\phi \in \A_{\pi}^\chi(\mu) ,$ we have 
 	$$ \int_\Omega( \phi + \chi)   \cdot \nabla \eta  =\int_\Omega \mu  \: \eta    +\langle \pi , \eta \rangle_{\Gamma_N} + \langle( \phi + \chi) \cdot \nu , g  \rangle_{\Gamma_D}. $$ 
Combining both equations we deduce that 
 $$  \int_\Omega F(.,\phi ) \: dx +   \int_\Omega F^*(.,\nabla \eta) \: dx   =  \int_\Omega \phi \cdot \nabla \eta.$$ 
Then, by using  the fact that $\phi \cdot \nabla \eta \leq F^*(.,\nabla\eta) + F(.,\phi ),$  we deduce that $F(.,\phi )+F^*(.,\nabla \eta)=\phi \cdot \nabla \eta$  and    $\phi \in \partial F^*(.,\nabla \eta)$ a.e. in $\Omega.$ Thus  $(\eta,\phi )$ solve the PDE \eqref{PDE1}.  For the   converse part,  it is not difficult to see first   that $$\I_{\pi,g}^\chi( \mu) \geq  \J_{\pi,g}^\chi( \mu) . $$
Then,  working with  $(\eta,\phi )$ being  a solution  of PDE \eqref{PDE1}, on proves easily the converse inequality and then concludes    $\I_{\pi,g}^\chi( \mu) =  \J_{\pi,g}^\chi( \mu) . $ 
 	
 \end{proofth}

  \bigskip

%
%
%

  \medskip 
  \begin{proofth}{Proof of Theorem \ref{tduality2}}  
As in the proof of Theorem \ref{tduality1}, we only need to compute 
$K_1^{*}.$      By definition,  for any $\eta \in W^{1,p}_{\Gamma_D}(\Omega),$  we have

	$$ K_0^*(\eta) =\max\left\{  \langle f,\eta \rangle_{\Omega} -  K_0(f ) \: :\:   f\in W^{-1,p'} _{\Gamma_D}(\Omega)\right\}    $$
$$ = \max_{f_0,\overline f,\phi }\Big\{  \underbrace{\int_\Omega  f_0   \:  \eta  -\int _\Omega \overline f \cdot \nabla  \eta  - \int_\Omega \beta(.,\rho) 	- \:  \int_\Omega F(.,\phi ) +\int_\Omega (\phi  +\chi )  \cdot \nabla \tilde g  \: dx  -\int_\Omega ( \mu -\rho) \: \tilde g  }_{H(f_0,\overline f,\rho,\phi )}  $$  
$$    \: :\:  (f_0,\overline f)\in L^{p'}(\Omega)\times L^{p'}(\Omega)^N,\: \phi  \in \A_{\pi}^\chi (\mu-\rho+f_0,\chi+\overline f),\: \rho\in L^{p'}(\Omega)\Big\}.    $$
Working as in the proof of Theorem \ref{tduality1}, we get  
\begin{eqnarray*} 
	H(f_0,\overline f,\rho,\phi ) 
	&=& \int_\Omega (\phi \cdot \nabla (\eta+\tilde g)-F(.,\phi ) ) +\int_\Omega (\rho\: (\eta+\tilde g) - \beta(.,\rho) ) \\  \\ 
	&  & - \int_\Omega \mu\: (\eta+\tilde g)   + \int_\Omega \chi\cdot \nabla (\eta+\tilde g)  -\langle \pi,\eta\rangle_{\Gamma_N}. \end{eqnarray*}
Then we deduce that 
 \begin{eqnarray*}
	K_1^*(\eta) 
& =& 	\max_{\rho,\phi }\Big\{  \int_\Omega (\phi \cdot \nabla (\eta+\tilde g)-F(.,\phi ) ) +\int_\Omega (\rho\: (\eta+\tilde g) - \beta(.,\rho) ) - \int_\Omega \mu\: (\eta+\tilde g)  \\  \\ &  & \hspace*{3cm}  + \int_\Omega \chi\cdot \nabla (\eta+\tilde g)   -\langle \pi,\eta\rangle_{\Gamma_N}      \: :\:    (\rho,\phi )\in L^{p'}(\Omega)\times L^{p'}(\Omega)^N \Big\}   \\  \\ 
	&=&  \max_{\phi  \in L^{p'}(\Omega)^N  }  \int_\Omega (\phi \cdot \nabla (\eta+\tilde g)-F(.,\phi ) ) +  \max_{\rho \in L^{p'}(\Omega)} \left\{ \int_\Omega (\rho\: (\eta+\tilde g) - \beta(.,\rho) )   \right. \\  \\ &  & \left. \hspace*{3cm}- \int_\Omega \mu\: (\eta+\tilde g)   + \int_\Omega  \chi\cdot \nabla (\eta+\tilde g)   -\langle \pi,\eta\rangle_{\Gamma_N}  \right\}   .
\end{eqnarray*}
Using the assumptions $(H1)$ and $(H4)$, we deduce    that  
\begin{eqnarray*}
	K_1^*(\eta)&= &\int_\Omega F^*(.,\nabla (\eta+\tilde g))  +\int_\Omega \beta^*(.,\eta+\tilde g)  - \int_\Omega \mu\: (\eta+\tilde g)   + \int_\Omega \chi\cdot \nabla (\eta+\tilde g)   -\langle \pi,\eta\rangle_{\Gamma_N} .
\end{eqnarray*}
Combining this with \eqref{perturbduality}, we obtain 
\begin{eqnarray*}
	-K_1(0)&=&  
	\min_{ \eta \in W^{1,p}(\Omega)}  \int_\Omega F^*(.,\nabla (\eta+\tilde g))  +\int_\Omega \beta^*(.,\eta+\tilde g)  - \int_\Omega \mu\: (\eta+\tilde g)   + \int\Omega \chi\cdot \nabla (\eta+\tilde g)   -\langle \pi,\eta\rangle_{\Gamma_N} \\  \\
	&=& \min_{ \eta \in W^{1,p}(\Omega),\: \eta_{/\Gamma_D}=g}  \int_\Omega  F^*(.,\nabla \eta)  +\int_\Omega \beta^*(.,\eta) - \int_\Omega \mu\: \eta    + \int \chi\cdot \nabla \eta    -\langle \pi,\eta\rangle_{\Gamma_N}. 
\end{eqnarray*}
Thus the duality  $\N_{\pi,g}^\chi (\mu) = D_{\pi,g}^\chi( \mu) . $ 
Now, thanks to Lemma \ref{Pexistprimal}, let us consider $(\rho,\phi )$ and $\eta$ be the solutions of $\N_{\pi,g}^\chi (\mu) =D_{\pi,g}^\chi( \mu) .$  We have 
$$   \int_\Omega  \mu\:\eta-\int_\Omega \chi\cdot \nabla\eta  - \int_\Omega F^*(.,\nabla\eta)   -\int_{\Omega } \beta^*(.,\eta)   + \langle \pi,\eta\rangle_{\Gamma_N}  $$ 
$$=  \int_\Omega   \beta(.,\rho) +  \int_\Omega F(.,\phi ) \: dx -  \langle (\phi +\chi)\cdot  \nu, g\rangle_{\Gamma_D} .$$ 
Since $ \phi  \in \A_{\pi}^\chi (\mu-\rho)  ,$ $\eta \in W^{1,p}(\Omega)$ and $\eta =g$ on $\Gamma_D,$  we also have   
$$\int_\Omega  \rho  \: \eta   + \int_\Omega( \phi + \chi)   \cdot \nabla \eta  =\int_\Omega \mu  \: \eta    +\langle \pi , \eta \rangle_{\Gamma_N}  + \langle (\phi +\chi)\cdot \nu, g \rangle_{\Gamma_D}   . $$ 
Combining both equation we get 
$$  \int_\Omega F(.,\phi )  + \int_\Omega F^*(.,\nabla\eta)  + \int_{\Omega } \beta^*(.,\eta)  +  \int_{\Omega } \beta (.,\rho)  = \int_\Omega \phi \cdot \nabla \eta + \int_\Omega \eta\: \rho. $$
This implies that $ F(.,\phi )  +  F^*(.,\nabla\eta) = \phi \cdot \nabla \eta$ and $\beta^*(.,\eta)  +    \beta (.,\rho) =\rho\: \eta$ a.e. in $\Omega.$  Thus  $(\rho,\eta,\phi )$ is a solution of the PDE \eqref{PDE2}.  The proof of the converse part follows in the same way of the proof of Theorem \ref{tduality1}. Indeed, first one sees directly that $D(\mu) \leq \N(\mu) .$  The, by working with the solution of  \eqref{PDE2} one proves     $D(\mu) = \N(\mu) .$

  \end{proofth}

  \bigskip 
  To end up this section, we prove the following results which will be useful for the sequel

  \begin{corollary}\label{CorEst}
  	For any  $\mu \in  L^{p'}(\Omega) $, $\pi  \in W ^{-1/p',p'}(\Gamma_N)$  and $g  \in W ^{1-1/p,p}_{00}(\Gamma_D)]$,   the problem  \eqref{PDE2} 
  	has a solution $(\rho,\eta,\phi )$ in the sense of : 	 $\rho  \in  L^{p'}(\Omega),$  $\eta \in W^{1,p} (\Omega),$  $\eta =g $ on $\Gamma_{D  },$   $\phi   \in L^{p'}(\Omega)^N,$  $\phi   \in \partial F^*  (.,\nabla \eta )$, 	$\rho \in \partial\beta^*(x,\eta ) $     a.e. in  $\Omega,$  and   
    $$\int_\Omega  \rho  \: \xi   \: dx   + \int_\Omega \phi   \cdot \nabla \xi  \: dx =\int_\Omega \mu  \: \xi  \: dx   -\int_\Omega  \chi \cdot \nabla \xi     -\langle  \pi , \xi\rangle_{\Gamma_N}    , \quad \forall \:  \xi\in W^{1,p}_{\Gamma_{ D} } (\Omega) .  $$      	Moreover,  the triplet $(\rho,\eta,\phi )$ satisfies 
  	\begin{equation} \label{eqextreme0}
  		\begin{array}{c}
  			\int_\Omega  \rho \: (\eta-\tilde g)+  \int_\Omega  \phi\cdot \nabla (\eta -\tilde g)  
  			=  	\int_\Omega  \mu\:(\eta-\tilde g) + \int_\Omega \chi\cdot \nabla(\eta-\tilde g)     + \langle \pi,\eta\rangle_{\Gamma_N}     .    
  		\end{array} 
  	\end{equation}	 
  	
  \end{corollary}
  
  \begin{proof}  The existence result is a direct consequence of Theorem \ref{tduality1}.  Let $(\rho,\eta,\phi )$ be a solution of the problem.   See that $\eta-\tilde g$ is an admissible test function. Thus \eqref{eqextreme0}.

  \end{proof}

%
  	\section{Steepest descent algorithm}  \label{SSteepest}
  	\setcounter{equation}{0}
  	
  	Thanks to Section 4,  the main application of the optimization problem $\I_{\pi,g}^\chi$ and $\N_{\pi,g}^\chi$ is to introduce a steepest descent algorithm  to study the evolution   problem of the type

  	\begin{equation}
  		\label{PDEevol}
  		\left\{\begin{array}{ll} 
  			\left.  \begin{array}{ll} 
  				\partial_t 	\rho  - \nabla\cdot ( \phi   +\rho  V )=  f    ,\quad        \\ 
  				\\  		\phi    =  \partial_\xi F  ^*(., \nabla \eta )  ,\:  		\rho  \in \partial \beta^*(x,\eta  )   \\   \\   
  			\end{array}\right\} 	\quad & \hbox{ in } Q, \\  \\
  			(\phi   +\rho  V )   \cdot \nu  =  \pi    & \hbox{ on }\Sigma_{N  }\\  \\
  			\eta   =  g    \quad    & \hbox{ on }\Sigma_{D  } \\  \\
  			\rho(0)=\rho_0& \hbox{ in }\Omega ,
  		\end{array}
  		\right. 	 
  	\end{equation} 
  	where we assume that $\rho_0\in L^{p'}(\Omega),$ with $V\in L^\infty(Q)$ and  $f=f_0+\nabla \cdot \overline f \in L^{p'}(0,T;W^{-1,p'}_{\Gamma_D}(\Omega)$ ; i.e. $(f_0,\overline f)\in L^{p'}(Q)\times L^{p'}(Q)^N.$ 
  	
  	\begin{definition}\label{Defsol1}
  			A weak solution for the problem \eqref{PDEevol}  is a triplet $(\rho,\eta,\phi )$ such that $\rho\in L^\infty(0,T;L^{p'}(\Omega))\cap W^{1,p'}(0,T; W^{-1,p'}_{\Gamma_D}(\Omega)),$ $\eta\in L^p(0,T;W^{1,p}(\Omega)), $  $\phi \in L^{p'}(Q),$  $\eta=g$ on $\Sigma_D,$ $\eta\in \partial \beta(x,\rho),$  $\nabla \eta  \in \partial F(x,\phi ),$ a.e. in $Q,$ and we have 
  	\begin{equation} \label{weaksol}
  		\frac{d}{dt}\int_\Omega \rho\: \xi\: dx + \int_\Omega   (\phi  +\rho\: V)  \cdot \nabla \xi \: dx = \langle f,\xi\rangle_{\Omega} + \langle \pi,\xi\rangle_{\Gamma_N} ,\quad \hbox{ in }\D'([0,T)),
  	\end{equation}
  	for any $\xi\in W^{1,p}_{\Gamma_D}(\Omega),$ with $\rho(0)=\rho_0$  in $\Omega.$   
  	\end{definition}

 \bigskip

    We proceed by discretizing time with discrete steps of size $\tau$, denoted as $t_0 = 0 < t_1 < \cdots < t_n =T$, with $t_i-t_{i-1}=\tau,$ for any $i=1,...n.$ Then, we  consider the limit as $\tau$ approaches zero in the sequence of piecewise constant curves 
  \begin{equation}\label{rhotau}
  	 \rho^\tau =\sum_{i=1}^n\rho^i\:  \chi_{]t_{i-1}  ,t_{i}]} \quad\hbox{ with } \rho^\tau(0)=\rho_0 ,\end{equation}    
  	where  $\rho^i$ is given, for each $i=1,....n,$  by 
\begin{equation}\label{coroptim1}
	   \rho^i = 	 	\hbox{argmin}_{\rho}\left \{\int_\Omega   \beta(.,\rho) +  \tau\: \int_\Omega F(.,\phi/\tau) \: dx - \tau\:   \langle (\phi -\rho^{i-1}V^i)\cdot  \nu, g\rangle_{\Gamma_D} \right. $$ $$\left.  \: :\:    \phi  \in  \A_{\pi}^{\tau \rho^{i-1}V^i+\tau \overline f^i} (\tau f_0^i +\rho^{i-1}-\rho),\:  \rho\in   L^{p'}(\Omega)  \right\}.
\end{equation}   
Here    $$ V^i:=  \frac{1}{\tau} \int_{t_{i-1}}^{t_i} V (t,.)\: dt,\:  f_0^i(.) :=   \frac{1}{\tau} \int_{t_{i-1}}^{t_i} f_0 (t,.)\: dt\hbox{ and } \overline f^i(.) :=   \frac{1}{\tau} \int_{t_{i-1}}^{t_i}\overline f(t,.)\: dt     \quad \hbox{ a.e. in }\Omega.  $$   
   
 In terms of $\I_{\pi,g},$ the optimization problem \eqref{coroptim1} may be written as follows 
\begin{equation} 
	\rho^i = 	 	\hbox{argmin}_{\rho\in   L^{p'}(\Omega)  }\left \{\int_\Omega   \beta(.,\rho) +   \tau \: \I_{\pi,g}^{\overline f+\rho^i\: V^i}   \left(\frac{ \tau f_0^i +\rho^{i-1}-\rho}{\tau}\right)     \right\} .\end{equation}

  	Thanks to Theorem \ref{tduality2}, we have

  	\begin{corollary}
  The problem \eqref{coroptim1} has a solution $(\rho_i,\phi_i) \in L^{p'}(\Omega) \times L^{p'}(\Omega)^{N}.$  Moreover, there exists $\eta_i\in W^{1,p}(\Omega)$ such that  $\eta^i=g$ on $\Gamma_{D  },$    $\rho^i \in \partial\beta^*(x, \eta^i ) $,   $\phi ^i\in \partial {F}^*  (.,\nabla \eta^i)$ a.e. in $\Omega,$ and the triplet $(\rho_i,\phi _i,\eta_i)$ satisfies in a weak sense the following system  PDE 
  	 	\begin{equation}
  		\label{PDEi}
  		\left\{\begin{array}{ll} 
  			\left.  \begin{array}{ll} 
  				\rho^i-\tau\: \nabla\cdot ( \phi ^i+\rho^{i-1} V^i +\overline f^i ) = \tau f^i_0 + \rho^{i-1} ,\quad     \\ 
  				\\  				\rho^i \in \partial \beta^*(x, \eta^i ),\quad \phi ^i   \in  \partial_\xi {F}^*(., \nabla \eta^i)   \\   \\     	 
  			\end{array}\right\} 	\quad & \hbox{ in } \Omega, \\  \\
  			(\phi ^i+\rho^{i-1} V^i +\overline f^i ) \cdot \nu  =  \pi    & \hbox{ on }\Gamma_{N  }\\  \\
  			\eta^i  =  g   ,\quad    & \hbox{ on }\Gamma_{D  } \: ; 
  		\end{array} \right.
  	\end{equation} 
i.e. 
  			$$\int_\Omega ( \rho^i-\rho  ^{i-1})\: \xi   \: dx+\tau\:  \int_\Omega (\phi ^i+\rho^{i-1} V^i + \overline f^i ) \cdot \nabla \xi   \: dx =\tau\:  \int_\Omega f^i_0 \: \xi    \: dx +\tau\:  \int_{\Gamma_{ N}} \pi\: \xi    \: dx , $$		
 for any $\xi\in W^{1,p}_{\Gamma_{ D} } (\Omega) .$  
  \end{corollary}
 \begin{proof}
 	It is enough to apply Theorem \ref{tduality2} with $\tilde F_\tau(A)=  \tau\: F(A/\tau)$, for any $A\in \RR^N.$  Then, it is not difficult to see that $\tilde F_\tau^*(A^*)=  \tau\:   F^*(\tau\: A^*)$, for any $A^*\in \RR^N.$ Moreover $A\in \partial \tilde F_\tau^*(A^*)$ if and only if $ A^*\in \tau \partial F(A) .$  
 \end{proof}

 \medskip 
 To let $\tau\to 0,$  we  consider  the sequences
    \begin{equation}\label{functionstau}\eta^\tau =\sum_{i=1}^n\eta^i\:  \chi_{]t_{i-1}  ,t_{i}]},\quad   \phi ^\tau =\sum_{i=1}^n\phi ^i\:  \chi_{]t_{i-1}  ,t_{i}]} ,\quad f^\tau_0 =\sum_{i=1}^nf^i\:  \chi_{]t_{i-1}  ,t_{i}]}, \quad \overline f^\tau=\sum_{i=1}^n\overline f^i\:  \chi_{]t_{i-1}  ,t_{i}]}  \end{equation}   
  and  
  	$$\tilde \rho^{\tau}(t) =\frac{(t-t_{i-1})\rho^{i} -(t-t_{i}) \rho^{i-1}}{ \tau},\quad \hbox{ for any }t\in ]t_{i-1}  ,t_{i}],\quad i=1,.... n .$$
By construction,   the applications $\tilde \rho^\tau,\: \rho^\tau,\: \eta^\tau$ and $\phi ^\tau$  satisfy  the following PDE 
  	\begin{equation} 
  		\left\{\begin{array}{ll} 
  			\left.  \begin{array}{ll} 
  				\partial_t 	\tilde \rho ^\tau - \nabla\cdot ( \phi   ^\tau+ \rho^\tau(.-\tau)V^\tau+ \overline f^\tau ) =  f  ^\tau_0       \quad   \\ 
  				\\  				\rho ^\tau \in \partial \beta^*(x,\eta ^\tau ),  \quad \phi   ^\tau \in  \partial_\xi F  ^*(., \nabla \eta ^\tau)  	 
  			\end{array}\right\} 	\quad & \hbox{ in } Q, \\  \\
  			(\phi    ^\tau +\rho^\tau(.-\tau)V^i  +\overline f^\tau  )\cdot \nu  =  \pi    & \hbox{ on }\Sigma_{N  }\\  \\
  			\eta ^\tau  =  g   .\quad    & \hbox{ on }\Sigma_{D  },
  		\end{array}
  		\right. 	 
  	\end{equation} 
  	where we extend $\rho^\tau$ by $\rho_0$ on $(-\infty,0).$ 
  	Our goal now  is to prove that, as $\tau\to 0,$   the $\tau-$approximation solutions  $\tilde \rho^\tau,\: \rho^\tau,\: \eta^\tau$   and  $\phi ^\tau$ converge to the solutions of the continuous  problem \eqref{PDEevol}.  We present the main results following this process in the subsequent sections, where we differentiate between different scenarios based on the structural assumptions imposed on  $F$ and $\beta.$

  \begin{remark} \label{RCLiggett}	  
  	
  	\begin{enumerate}
  		\item 
   Note that  the discrete approximation   \eqref{PDEi} is widely used in the context of nonlinear semigroup theory (at least in the case $V\equiv 0$). The time step approximation $\rho^\tau$  is typically referred to as the Crandall-Liggett approximation (cf. \cite{CL}). Its convergence can be attributed to the $L^1-$accretivity of the fundamental operator that drives the dynamic. This approach of ''monotonicity'' enables to come with strong compactness  of $\rho^\tau$ and $\tilde \rho^\tau.$  We do not use this approach  in this paper since the $L^1-$accretive  property is hard to obtain  for general $V$ and $F,$  even if we do believe that it should be a nice issue for the proof of existence (and uniqueness) of a solution.   On can see for instance \cite{IgShaw} and \cite{IgPME} for  linear diffusion with Hele-Shaw   and PME non-linearity respectively.  	Consequently, the approach presented in this paper can be viewed as a fresh perspective on the discrete approximation process typically employed in Crandall-Liggett theory for nonlinear evolution problems in $L^1(\Omega).$ 
  
   \item Observe that in the situation where $V\equiv 0,$ the steepest descent algorithm through minimum flow we employ in \eqref{coroptim1} directly derives from Euler-Implicit discretization for the evolution problem \eqref{PDEevol}. It is worth noting that this is not the case for steepest descent algorithm through Wasserstein metric, even though both methods can capture the same continuous dynamics (at least in the case of homogeneous Neumann boundary conditions and  $F$ of $p-$Laplacian type). This observation may be linked to the fact that the linearized regime of Wasserstein distance is closely associated with dual Sobolev norm (cf. Section 7.6 of \cite{Vbook}), and the fact that   our approach is intimately connected to dual Sobolev metric gradient flow in the homogeneous case (see Section \ref{RemH1}).

   	\end{enumerate}	 
  	 
  	   \end{remark}

    	\subsection{Main results }

 \subsubsection{The case where $V\equiv 0$}
 
Assuming  $V\equiv 0,$ the evolution problem \eqref{evol0} corresponds to   
 \begin{equation}\label{FVnull}
 	\left\{\begin{array}{ll} 
 		\left.  \begin{array}{ll} 
 			\partial_t 	  \rho   - \nabla\cdot  ( \phi  +\overline f)  =  f_0         \quad   \\ 
 			\\  				\rho  \in \partial \beta^*(.,\eta   ),  \quad \phi    \in  \partial_\xi F  ^*(., \nabla \eta  ) 
 		\end{array}\right\} 	\quad & \hbox{ in } Q, \\  \\
 		(\phi  + \overline f) 	\cdot \nu  =  \pi    & \hbox{ on }\Sigma_{N  }\\  \\
 		\eta  =  g     & \hbox{ on }\Sigma_{D  }\\ \\
 		\rho(0)=\rho_0  & \hbox{ in }\Omega ,
 	\end{array}
 	\right. 	 
 \end{equation} 
 where $(f_0,\overline f) \in L^{p'}(Q)\times L^{p'}(Q) ^N.$

 \begin{theorem}\label{texistVnull} 
 	Under  the assumptions $(H1)-(H4),$ for any $\rho_0\in L^{p'}(\Omega) $    and $ ( f_0, f)  \in  L^{p'}(Q)\times  L^{p'}(Q)^N ,$ the problem \eqref{FVnull} 	has a weak solution $(\rho,\eta,\phi ).$  
 \end{theorem}

\subsubsection{Diffusion-Transport problem with linear diffusion}
For  linear diffusion, we select  quadratic $F$ as follows : 
\begin{equation}\label{Flinear}
	F(x,A)=\frac{k(x)}{2}\vert A\vert^2,\quad \hbox{ for a.e. }x\in \Omega\hbox{ and }A\in \RR^N,
\end{equation}
where  $ k\in L^\infty(\Omega),$ is such that $0<\mathrm{Infess}_\Omega(k).$  In this case, the evolution problem  reads 	 
 	\begin{equation}\label{PDElinear}
 			\left\{\begin{array}{ll} 
 			\partial_t 	  \rho   - \nabla\cdot ( k\: \nabla \eta +\rho\: V+\overline f) =  f_0,          \quad  	\rho  \in \partial \beta^*(x,\eta  )  
 			\quad & \hbox{ in } Q, \\  \\
 			(k\: \nabla \eta   +\rho\: V  +\overline f  )\cdot \nu  =  \pi    & \hbox{ on }\Sigma_{N  }\\  \\
 			\eta  =  g     & \hbox{ on }\Sigma_{D  }\\ \\
 			\rho(0)=\rho_0  & \hbox{ in }\Omega.
 		\end{array}
 		\right. 	 
 \end{equation}

\begin{theorem}\label{texistlinear} 
Assume $(H1)-(H4)$ are  fulfilled.  For any $\rho_0\in L^{2}(\Omega),$  $V\in L^\infty(Q)$ and  $( f_0, f)  \in  L^{2}(Q)\times  L^{2}(Q)^N ,$  
\end{theorem}

\subsubsection{Diffusion-Transport problem with nonlinear diffusion}
For general $F,$ not necessary quadratic as in Theorem \ref{texistlinear},  we need to work with additional assumptions on $\beta.$  For technical reason, we assume in this case the following condition on $\beta$ :    

 	\begin{itemize}
 	\item[(H5)]    for a.e. $x\in \Omega,$  the application $r\in\RR \to   \beta^*(x,r)$  is  differentiable.
 	\end{itemize} 

\begin{theorem}\label{texistgen} 
Under the assumptions  $(H1)-(H5)$, for any $\rho_0\in L^{p'}(\Omega),$  $V\in L^\infty(Q)$ and $(f_0,\overline f) \in L^{p'}(\Omega)^N \times L^{p'}(\Omega)) ,$ the problem \eqref{PDElinear} 	has a weak solution $(\rho,\eta,\phi ).$   
\end{theorem}

  	\subsection{Proofs of Theorem \ref{texistVnull}, Theorem \ref{texistlinear} and Theorem \ref{texistgen} }

  Building upon the notation established at the beginning of this section, our primary goal is  the (weak) compactness of the sequences $\rho^\tau,$ $\tilde \rho^\tau,$ $\eta^\tau$ and $\phi ^\tau.$ 
   Next, we need to establish the connections between the limits of $\rho^\tau$ and $\eta^\tau$, and also between the limits of $\eta^\tau$ and   $\phi ^\tau.$  The proofs follow from the subsequent lemmas.

  	\begin{lemma}\label{Lrhotproperty}
  		For any $t\in [0,T) ,$ we have 
  	\begin{equation}\label{rhotproperty}
  	\begin{array}{c}
  	\frac{d}{dt}	\int_\Omega ( \beta(.,\rho^{\tau}) -   \rho^{\tau}\: \tilde g  )+ \int_\Omega  \phi ^{\tau}  \cdot \nabla(\eta^{\tau} -\tilde g)\leq 	 \langle f^{\tau},    \eta^{\tau}-\tilde g \rangle_{\Omega}  \\  \\  
  	\hspace*{3cm}	-   \int_\Omega  \rho^{\tau}(t-\tau) V^{\tau}   \cdot \nabla(\eta^{\tau}-\tilde g)     +   \langle \pi,\eta^{\tau}\rangle_{\Gamma_N},\quad \hbox{ in }\D'([0,T)  .  
  	\end{array}
  \end{equation}	   		 
  	\end{lemma}
  	\begin{proof} 
  Thanks to Corollary \ref{CorEst}, we know that 
  	\begin{equation} \label{eqextremei}
  	\begin{array}{c}
  		\int_\Omega  \rho^{i} \: (\eta^{i}-\tilde g)+  \tau\: \int_\Omega  \phi^{i}\cdot \nabla (\eta^{i} -\tilde g)  
  		=  	\int_\Omega  \rho^{i-1} \: (\eta^{i}-\tilde g) +	\tau \: \langle  f^{i},\eta^{i}-\tilde g\rangle_{\Omega} \\  \\  - \int_\Omega \rho^{i-1}\: V^{i}\cdot \nabla(\eta^{i}-\tilde g)     + \langle \pi,\eta^{i}\rangle_{\Gamma_N}     .    
  	\end{array} 
  \end{equation}
Since $(\rho^{i}-\rho^{i-1})\eta^{i}\geq \beta(.,\rho^{i})-  \beta(.,\rho^{i-1}),$ this  implies that 	 
  		\begin{equation}\label{eqextreme2}
  		\begin{array}{c}
  			\int_\Omega  ( \beta(.,\rho^{i}) - 	  \rho^{i} \: \tilde g ) -  	\int_\Omega  ( \beta(.,\rho^{i-1}) - 	  \rho^{i-1} \: \tilde g ) + \tau  \int_\Omega  \phi ^{i}  \cdot \nabla(\eta^{i} -\tilde g)\leq    \tau \langle  f^{i}   ,\eta^{i}-\tilde g\rangle_{\Omega}  \\  \\ 
  			-\tau  \int_\Omega  \rho^{i-1}V^{i}   \cdot \nabla(\eta^{i}-\tilde g)     +\tau   \langle \pi,\eta^{i}\rangle_{\Gamma_N}   .  
  		\end{array} 
  	  	\end{equation}	
  	  Thus \eqref{rhotproperty}. 
  	     	\end{proof}
  	
  	\begin{lemma}\label{Propcompacttau}
  		We have 
  		\begin{enumerate}
  			\item There exists a constant $C=C(p,N,\Omega,\Vert V\Vert_\infty)$ such that :  
  			\begin{equation}\label{Estpple}
  				\begin{array}{c}
  					C\left\{ 	\int_\Omega  (\vert  \rho^{\tau}(t)\vert -M)^{+p'}  +\int_0^T\!\! \int_\Omega  \vert \nabla\eta^{\tau}\vert^p \right\} \leq   	B\left( 1+      \frac{  T\:     \Vert V^{\tau}\Vert_\infty^{p'}}{C} \:  \exp\left( T  \Vert V^{\tau}\Vert_\infty^{p'}/C\right)  \right) ,
  				\end{array}  
  			\end{equation}	 
  			where $B$ is given by 
  			$$ B = \int_\Omega   \beta(.,\rho_0) +	\int_\Omega  \vert \rho_0-M\vert^{p'}      +  	\int_\Omega   \vert \tilde g\vert^p     +  T\: \Vert V\Vert_\infty^{p'} \int_\Omega \vert  \nabla\tilde g\vert^{p}      
  			+T \:  \Vert \pi\Vert_{W^{-1/p',p'}(\Gamma_N)}^{p'}  $$  $$+ \tau  \int_\Omega  (\vert \rho_0 \vert  -M)^{+p}  + \sup_{0<\tau<1}\int_0^T\!\! 	\int_\Omega (  \vert  f^{\tau}_0  \vert ^{p'} +  \vert \overline  f^{\tau}  \vert ^{p'} ) .    $$ 		  
  	 \item  The sequences $\rho^\tau$  and $\tilde \rho^\tau$ are bounded in $L^\infty\left(0,T; L^{p'}(\Omega) \right),$     $\eta^\tau$   is bounded in  $L^p\left(0,T; W ^{1,p} (\Omega)  \right)$ and $\phi ^\tau$   in bounded in $ L^{p'}(Q)^N.$  
  			\item The sequence $\partial_t \tilde \rho^\tau$ is  bounded in   $L^{p'}\left(0,T;W ^{-1,p'}_{\Gamma_D}(\Omega)  \right).$   
  		\end{enumerate}
  	\end{lemma}
  	\begin{proof} 
 Using Lemma \ref{Lrhotproperty},  the assumption $(H1)$ and $(H4),$ and the fact that $ \phi ^{\tau}  \cdot \nabla(\eta^{\tau}-\tilde g)  \geq F^*(\nabla \eta^\tau)-F^*(\nabla \tilde g) ,$ we get  
  		\begin{equation}
  		\begin{array}{c}
  		C_3\: 	\int_\Omega  ( \vert \rho^{\tau}(t)\vert -M)^{+p'}  +C_1\: \int_0^t\!\! \int_\Omega  \vert \nabla\eta^{\tau}\vert^p \leq  	\int_\Omega   \beta(.,\rho_0)  - 	\int_\Omega  \rho_0 \: \tilde g  +	\int_\Omega  \rho^{\tau}(t) \: \tilde g   +  T \int_\Omega F^*(., \nabla\tilde g)        \\  \\ 
  		+ \int_0^T\langle f^{\tau} ,\eta^{\tau}-\tilde g\rangle_{\Omega } -\int_0^T\!\!   \int_\Omega  \rho^{\tau}(t-\tau) V^{\tau}   \cdot \nabla(\eta^{\tau}-\tilde g)     +\int_0^T\!\!   \langle \pi,\eta^{\tau}\rangle_{\Gamma_N},   
  		\end{array}
  	\end{equation}	  
   for any  $t\in [0,T) ,$  and  then 
  	 	\begin{equation}
  	 	\begin{array}{c}
  	 	C_3\:  	\int_\Omega  ( \vert \rho^{\tau}(t)\vert -M)^{+p'}  +C_1\: \int_0^T\!\! \int_\Omega  \vert \nabla\eta^{\tau}\vert^p \leq  	\int_\Omega   \beta(.,\rho_0)  +	\int_\Omega  (M-\rho_0) \: \tilde g  +	\int_\Omega  ( \vert \rho^{\tau}(t)\vert -M)^+ \: \tilde g   \\  \\     +  T \int_\Omega F^*(., \nabla\tilde g)   + \int_0^T\langle    f^{\tau} ,\eta^{\tau}-\tilde g\rangle_{\Gamma_N}     
  	 		-\int_0^T\!\!   \int_\Omega  \rho^{\tau}(t-\tau) V^{\tau}   \cdot \nabla(\eta^{\tau}-\tilde g)     +\int_0^T\!\!   \langle \pi,\eta^{\tau}\rangle_{\Gamma_N}, 
  	 	\end{array}
  	 \end{equation}	 
   for any $t\in [0,T).$ 
  	 Using Young and Poincaré inequalities,  we can find  a constant $C=C(p,N,\Omega)>0$ such that  
  	\begin{equation}\label{Esttau}
  	\begin{array}{c}
  	C\left\{ 	\int_\Omega  (\vert  \rho^{\tau}(t)\vert -M)^{+p'}  +\int_0^T\!\! \int_\Omega  \vert \nabla\eta^{\tau}\vert^p \right\} \leq   	\int_\Omega   \beta(.,\rho_0) +	\int_\Omega  \vert \rho_0-M\vert^{p'}      +  	\int_\Omega   \vert \tilde g\vert^p   +  T \int_\Omega \vert  \nabla\tilde g\vert^{p}        \\  \\   
  	     +T \:  \Vert \pi\Vert_{W^{-1/p',p'}(\Gamma_N)}^{p'}  + \int_0^T\!\! 	\int_\Omega  ( \vert  f^{\tau}_0  \vert ^{p'}    + \Vert \overline f\Vert^{p'}) 	+\int_0^T\!\!   \int_\Omega  \rho^{\tau}(t-\tau) V^{\tau}   \cdot \nabla(\eta^{\tau}-\tilde g)   , \quad \hbox{ for any } t\in [0,T),  
  	\end{array}
  \end{equation}	 	 
 where we use also the fact that  $  F^*(., \nabla\tilde g) \leq C'_2\Vert \nabla \tilde g\vert^{p} .$ On the other hand, we see that the last term satisfies 
  	\begin{equation}
   	\begin{array}{c}
   		\int_0^T\!\!   \int_\Omega  \rho^{\tau}(t-\tau) V^{\tau}   \cdot \nabla(\eta^{\tau}-\tilde g) \leq    	\int_0 ^T\!\!   \int_\Omega  (\vert \rho^{\tau}(t-\tau) \vert  -M)^+  \vert V^{\tau}\vert (\vert \nabla \eta^{\tau}\vert + \vert \nabla \tilde g\vert) \\  \\  + M  \: 	\int_0^T\!\!   \int_\Omega  \vert V^{\tau}\vert \: (\vert \nabla \eta^\tau\vert +\vert \nabla \tilde g\vert) .
   	\end{array}
   \end{equation}	 	
   Using again Young inequality and the fact that  $$	\int_0 ^T\!\!   \int_\Omega  (\vert \rho^{\tau}(t-\tau) \vert  -M)^{+p} \leq 	\int_0^T\!\!   \int_\Omega  (\vert \rho^{\tau}(t) \vert  -M)^{+p} + \tau  \int_\Omega  (\vert \rho_0 \vert  -M)^{+p} ,$$  for any $\eps>0$ small enough,  we can find  a constant that we denote   by $C_\epsilon=C_\epsilon(p,N,\Omega,\epsilon)>0,$ such that 
 	\begin{equation}
 	\begin{array}{c}
 	C_\epsilon \int_\tau ^T\!\!   \int_\Omega  \rho^{\tau}(t-\tau) V ^{\tau}  \cdot \nabla(\eta^{\tau}-\tilde g) \leq  \Vert V^{\tau}\Vert_\infty^{p'} \: \left\{  	 \int_0^T\!\!   \int_\Omega     (\vert  \rho^{\tau}(t)\vert -M)^{+p'} +  \epsilon \:  \int_0^T\!\!   \int_\Omega  \vert \nabla \eta^{\tau} \vert^p   \right. \\  \\ \left.     +       \epsilon \: T  \:     \int_\Omega  \vert \nabla \tilde g \vert^p + \tau  \int_\Omega  (\vert \rho_0 \vert  -M)^{+p} \right\} .
 	\end{array}
 \end{equation}	 	 
 Working with a fixed small $\epsilon,$ and combining this with \eqref{Esttau},  we can find  a constant   $C=C(p,N,\Omega,\Vert V\Vert_\infty )>0,$ such that   
   	\begin{equation}\label{Esttau1}
   	\begin{array}{c}
   			C\left\{ 	\int_\Omega  (\vert  \rho^{\tau}(t)\vert -M)^{+p'}  +\int_0^T\!\! \int_\Omega  \vert \nabla\eta^{\tau}\vert^p \right\} \leq   	B+          \Vert V^{\tau}\Vert_\infty^{p'} \: 	 \int_0^T\!\!   \int_\Omega     (\vert  \rho^{\tau}(t)\vert -M)^{+p'}    .
	\end{array}  
   \end{equation}	 	    
  This implies that 
  \begin{equation}
  	 C\:  \int_\Omega  (\vert  \rho^{\tau}(t)\vert -M)^{+p'} \leq  B\: \exp\left( T  \Vert V^{\tau}\Vert_\infty^{p'}/C\right) 
  \end{equation}
and then,   \eqref{Estpple} follows by \eqref{Esttau1}.  As to  $	\partial_t 	\tilde \rho ^\tau,$ we recall that 
  $$	\partial_t 	\tilde \rho ^\tau =     f_0 ^\tau +\nabla \cdot (\rho^\tau(.-\tau)V^{\tau} + \phi   ^\tau+   \overline f^\tau)   $$
with   $  (\phi    ^\tau+\rho^\tau(.-\tau)V^{\tau}  +\overline f^\tau)  \cdot \nu  =  \pi $  on  $\Sigma_{N  }.$ Using  \eqref{Estpple} and $(H2),$   one sees easily that  $\phi    ^\tau+\rho^\tau(.-\tau)V^{\tau}  +\overline f^\tau $ is bounded in $L^{p'}(Q).$   This implies that $	\partial_t 	\tilde \rho ^\tau $ is bounded in $L^{p'}\left(0,T;W ^{-1,p'}_{\Gamma_D}(\Omega) \right).$

  	\end{proof}

  	\begin{lemma} \label{lcondsol1}
  		There exists   $\rho\in L^{\infty}\left(0,T; L^{p'}(\Omega) \right)$,  $\eta\in L^p\left(0,T; W ^{1,p}_{\Gamma_D}(\Omega)  \right),$  $\phi \in L^{p'}(Q)^{N},$ and  subsequences that we denote by again $\rho^\tau ,$  $\tilde \rho^\tau $  and $\eta^\tau,$  such that 
  		\begin{equation}\label{convrho}
  			\rho^\tau \to \rho,\quad \hbox{ in } L^{p'}(Q)  -\hbox{weak}, 
  		\end{equation} 
  			\begin{equation}\label{convtrho}
  			\tilde \rho^\tau \to \rho,\quad \hbox{ in } L^{p'}(Q)  -\hbox{weak}, 
  		\end{equation}  
  		\begin{equation}\label{conveta}
  			\eta^\tau \to \eta,\quad \hbox{ in } L^p\left(0,T; W ^{1,p}_{\Gamma_D}(\Omega)  \right)-\hbox{weak} 
  		\end{equation} 
and \begin{equation}\label{convphi}
	\phi ^\tau \to \phi ,\quad \hbox{ in } L^{p'}(Q)^{N}-\hbox{weak} .
\end{equation} 
  	Moreover,  we have 
  	\begin{enumerate} 
  		\item   	 $\eta\in \partial \beta(.,\rho),$ a.e. in $Q$  
  		\item Under the assumption $(H5),$ we have 
  			\begin{equation}\label{convrhostrong}
  			\rho^\tau \to \rho,\quad \hbox{ in } L^{p'}(Q) . 
  		\end{equation}  
\item   If  \begin{equation}\label{condphi}
	 \phi \in \partial F^*(.,\nabla \eta), \quad \hbox{ a.e. in }Q, 
\end{equation}
  then the  triplet  $(\rho,\eta,\phi )$ is a solution of the problem  \ref{FVnull}.  
  
    	\end{enumerate}
    		\end{lemma}

    		Thanks to Lemma \ref{Propcompacttau},  the weak compactness of the approximations  $\tilde \rho^\tau,$ $\rho^\tau,$ $\eta^\tau$ and $\phi ^\tau$ is clear. Now, in order to connect  the limit of $\rho^\tau,$ and $\eta^\tau$  through the  non-linearity $\partial \beta,$ we use the   very weak Aubin's result du to Andreanov and al. (cf. Theorem A.1 in the appendix, the proof may be found  in \cite{ACM} or \cite{Moussa}).

    \bigskip 		
    		
  	\begin{proofth}{Proof of Lemma \ref{lcondsol1}}
  		The weak compactness of $\rho^\tau$, $\tilde \rho^\tau$   and $\eta^\tau$  follows clearly  by Lemma \ref{Propcompacttau}.  Moreover, since $$\tilde \rho^k(t)-\rho^\tau(t) =  (t-t_i)\: \partial_t \tilde \rho^k(t)     ,\quad \hbox{ for any }t\in ]t_{i-1},t_{i}],\: i=1,... n.  $$  
  	and $\partial_t \tilde \rho^\tau$ is bounded in $L^{p'} (0,T;W ^{-1,p'}_{\Gamma_D}(\Omega)^*  )$, we deduce that $\rho^\tau$ and $\tilde \rho^\tau$ have the same $L^{p'}(Q)$-weak limits.   Now, having in mind Theorem \ref{ThBCM}, we combine  \eqref{convrho}, \eqref{conveta} with the fact that   $\partial_t \tilde \rho^\tau$ is bounded in $L^{p'} (0,T;W ^{-1,p'}_{\Gamma_D}(\Omega)^*  )$   to deduce first that, as $\tau\to 0,$ we have  
  		\begin{equation}
  			\int_0^T\!\!\int_\Omega   \rho ^\tau \: \eta ^\tau \: \varphi\: dxdt \to 	\int_0^T\!\!\int_\Omega \rho \: \eta  \: \varphi\: dxdt  \quad \hbox{ for any }\varphi\in \D ( Q). 
  		\end{equation} 
  		Then, using the fact that  
  		$  \eta^\tau \in \partial \beta(   \rho^\tau),$  we deduce that   
  		\begin{equation}
  		    \eta  \in \partial \beta( .,\rho),\quad \hbox{ a.e. in }Q. 
  		\end{equation}
 	Thanks to \eqref{rhotproperty} and \eqref{Estpple}, we see that $\int F^*(.,\phi ^\tau)$ is bounded. Combining this with  with the assumption $(H2)$,   we deduce that   $\phi ^\tau$   bounded in $L^{p'}(Q),$ and \eqref{convrhostrong} follows.      
 Remember that, for any $t\in (0,T)$ and $\xi\in W^{1,p}_{\Gamma_D}(\Omega),$ we have  
 $$  \int _\Omega 	\partial_t 	\tilde \rho ^\tau (t) \xi +\int_\Omega \phi   ^\tau (t) \cdot \nabla \xi  = \langle  f  ^\tau(t) ,\xi\rangle_{\Omega} - \int_\Omega  (\rho^\tau(t-\tau)V^\tau(t)\cdot \nabla \xi + \langle \pi,\xi\rangle_{\Gamma_N}  ,$$
and then 
\begin{equation}\label{eqtau1}
\begin{array}{c} 	  \frac{d}{dt} \int _\Omega  	\tilde \rho ^\tau (t) \xi +\int_\Omega (\phi   ^\tau (t) +  \rho^\tau(t) V^\tau(t)) \cdot \nabla \xi  = <f  ^\tau(t) , \xi>  +  \langle \pi,\xi\rangle_{\Gamma_N}  \\ \\ + \underbrace{\int_\Omega (\rho^\tau(t) - \rho^\tau(t-\tau) )V^\tau(t) \cdot \nabla \xi }_{J_\tau(t)} ,\quad \hbox{ in }\D'([0,T))  .
\end{array} 
\end{equation}   
Using the fact that  $\rho^\tau$ is bounded in $L^{p'}(Q)$ and $V^\tau$ is relatively compact in $L^p(Q),$ we see that  $J_\tau\to 0$ weakly. 
So, passing to the limit in \eqref{eqtau1}, we deduce that   the triplet $(\rho,\eta,\phi )$ satisfies  $\rho  \in  L^\infty(0,T;L^{p'}(\Omega)),$ $\partial_t \rho \in  W^{1,p'}(0,T; W^{-1,p'}_{\Gamma_D}(\Omega)) ,$ $\rho(0)=\rho_0$  in $\Omega,$  $\eta \times L^p(0,T;W^{1,p}(\Omega)),$  $\eta=g$ on $\Sigma_D,$ $\eta\in \partial \beta(x,\rho),$ a.e. in $Q,$  $ \phi  \times  L^{p'}(Q)$   and we have 
\begin{equation}  
	\frac{d}{dt}\int_\Omega \rho\: \xi\: dx + \int_\Omega   (\phi  +\rho\: V)  \cdot \nabla \xi \: dx = \langle f,\xi\rangle_{\Omega} + \langle \pi,\xi\rangle_{\Gamma_N} ,\quad \hbox{ in }\D'([0,T)), \: \forall \xi\in W^{1,p}_{\Gamma_D}(\Omega).
\end{equation}
  At last it is clear if   $\phi = \partial F(.,\nabla \eta),$ then  the triplet is definitely a solution of the problem  \eqref{PDEevol}  in a weak sense. 
  	\end{proofth}

  	\bigskip 
  	
  	\begin{proofth}{Proof of Theorem \ref{texistlinear}}  Clearly, the proof of  Theorem \ref{texistlinear} follows by Lemma \ref{lcondsol1}. Indeed, since in this case $\phi ^\tau=k(x) \: \nabla \eta^\tau$ and $\nabla \eta^\tau \to \nabla \eta$ in $L^{2}(Q)\hbox{-weak},$ then \eqref{condphi} is obviously fulfilled.  
  		 	\end{proofth}

  	\bigskip 
  	The main difficulty in proving Theorems \ref{texistVnull} and \ref{texistgen} lies in the  justification of $\phi \in \partial F(x,\nabla \eta).$ To address this, we employ a Minty-type argument and establish a sufficient condition that encompasses weak/strong convergences of $\rho^\tau,$ $\nabla \eta^\tau$ and $V^\tau.$

  	\begin{lemma} \label{lcondsol2} 
  	Under the assumptions of Lemma \ref{lcondsol1}, if 	\begin{equation}\label{condsuff1}
  	 \int_0^T\!\! \int_\Omega  	\rho\: V\cdot \nabla \eta \: dtdx \leq 	\limsup_{\tau \to 0}\int_0^T\!\! \int_\Omega  	\rho^\tau(t-\tau)\: V^\tau\cdot \nabla \eta^\tau\: dtdx, 
  		\end{equation}
then    \eqref{condphi} is fulfilled.  
  		\end{lemma}
  	\begin{proof} Having in mind $\phi ^\tau \in \partial F^*(.,\nabla \eta^\tau)$ and the weak convergence of $\phi ^\tau$ and $\nabla \eta^\tau$ in $L^{p'(Q)^N,}$ we use standard monotonicity arguments of Minty's type to prove this lemma. Indeed, it is enough to prove that $	\liminf_{\tau \to 0}	\int\!\!\!\int_Q \phi ^\tau \cdot \nabla (\eta^k-\eta)\leq 0.$ This is equivalent to prove that 
  		 	\begin{equation}\label{Minarg}
  				\liminf_{\tau \to 0}	\int\!\!\!\int_Q \phi ^\tau \cdot \nabla \eta^k \: dtdx\leq \int\!\!\!\int_Q \phi   \cdot \nabla  \eta\: dtdx. 
  			\end{equation}  
  	First, one sees that letting $\tau\to 0,$ in \eqref{rhotproperty}, and using the l.s.c. of the convex function $\beta$ with respect to the $L^{p'}(Q)$-weak topology, we have 
  		\begin{equation} 
  			\begin{array}{c}
  				\liminf_{\tau \to 0}\int_0^T\!\! \int_\Omega  \phi ^{\tau}  \cdot \nabla\eta^{\tau} \leq   \int_0^T\!\! \int_\Omega  \phi  \cdot \nabla\tilde g -\int_\Omega   \beta(.,\rho(t)) +	\int_\Omega  \rho(t) \: \tilde g   + 	\int_\Omega   \beta(.,\rho_0)  - 	\int_\Omega  \rho_0 \: \tilde g     \\  \\ + \int_0^T\langle  f   ,\eta -\tilde g\rangle_{\Omega} 
  				+\int_0^T\!\!   \langle \pi,\eta \rangle_{\Gamma_N}  	-\limsup_{\tau \to 0}  \int_0^T\!\!   \int_\Omega  \rho^{\tau}(t-\tau) V^\tau   \cdot \nabla(\eta^{\tau}-\tilde g)      .  
  			\end{array}
  		\end{equation}	   
Assuming  \eqref{condsuff1}, we obtain in one hand 
   		\begin{equation} \label{eq5}
  		\begin{array}{c}
  			\liminf_{\tau \to 0}\int_0^T\!\! \int_\Omega  \phi ^{\tau}  \cdot \nabla\eta^{\tau} \leq   \int_0^T\!\! \int_\Omega  \phi  \cdot \nabla\tilde g -\int_\Omega   \beta(.,\rho(t)) +	\int_\Omega  \rho(t) \: \tilde g   + 	\int_\Omega   \beta(.,\rho_0)  - 	\int_\Omega  \rho_0 \: \tilde g     \\  \\ + \int_0^T\langle   f ,\eta -\tilde g\rangle_{\Omega}  
  			+\int_0^T\!\!   \langle \pi,\eta \rangle_{\Gamma_N}  	- \int_0^T\!\!   \int_\Omega  \rho V   \cdot \nabla(\eta -\tilde g)     .  
  		\end{array}
  	\end{equation}	   
On the other hand, since  $ \partial_t 	 \rho  =   \nabla\cdot  ( \phi   +\rho\: V+\overline f) + f_0     $ in $Q$ and $	(\phi +\rho\: V+\overline f)   \cdot \nu  =  \pi $ on $\Sigma_{N  },$ by using  Lemma \ref{ltech2}, we have 
  		\begin{equation} 
  		\begin{array}{c}
  		 \int_0^T\!\! \int_\Omega  \phi   \cdot \nabla\eta =    \int_0^T\!\! \int_\Omega  \phi  \cdot \nabla\tilde g -\int_\Omega   \beta(.,\rho(t)) +	\int_\Omega  \rho(t) \: \tilde g   + 	\int_\Omega   \beta(.,\rho_0)  - 	\int_\Omega  \rho_0 \: \tilde g     \\  \\ + \int_0^T\langle   f   ,\eta -\tilde g\rangle_{\Omega}  
  			+\int_0^T\!\!   \langle \pi,\eta \rangle_{\Gamma_N}  -\int_0^T\!\!   \int_\Omega  \rho V   \cdot \nabla(\eta -\tilde g)     .  
  		\end{array}
  	\end{equation}	   
  	 Combining this with \eqref{eq5}, we deduce \eqref{Minarg}.  Thus the result of the lemma. 
 \end{proof}

  	\bigskip 
  	  	\begin{proofth}{Proof of Theorem \ref{texistVnull}}  Since we are assuming here that $V\equiv 0,$  \eqref{condsuff1} is obviously true and the proof follows directly from Lemma \ref{lcondsol2} .  
  	   		\end{proofth}

  	\bigskip
  	\begin{proofth}{Proof of Theorem \ref{texistgen}}   Using  Lemma \ref{convrhostrong}, we know that under the assumption $(H5),$  $\rho^\tau$ converges in $L^{p'}(Q).$   This implies that \eqref{condsuff1} is fulfilled.   Indeed, it is enough  to rewrite  
  		$$  \int_0^T\!\! \int_\Omega  	\rho^\tau(t-\tau)\: V^\tau\cdot \nabla \eta^\tau =	   \underbrace{ \int_0^T\!\!   \int_\Omega  (\rho^{\tau}(t-\tau) - \rho^{\tau}(t)) V^\tau   \cdot \nabla  \eta^{\tau}   }_{J^1_\tau} +  \underbrace{\int_0^T\!\!   \int_\Omega  \rho^{\tau}(t) V^\tau   \cdot \nabla \eta^{\tau} }_{J^2_\tau(t)} .$$
  		Then using  the fact that $V^\tau$ and $\nabla \eta^\tau$ are bounded in $L^\infty(Q)$ and $L^{p}(Q),$ respectively, we have 	      $J^1_\tau\to 0$,   and 
  		$$	\lim_{\tau\to 0} J _\tau =  \int_0^T\!\!   \int_\Omega  \rho V   \cdot \nabla \eta  .$$

  		\end{proofth}

  \section{Comments, remarks and extensions}\label{SRemarks}
  	  \setcounter{equation}{0}
    
 \subsection{$W^{-1,p'}_{\Gamma_D}(\Omega)$-Gradient flow and $p'-$curves of maximal slope} \label{RemH1}
Despite the fact that \eqref{PDEi} coincides with Euler-Implicit discretization in $L^1(\Omega),$   the primary differential operator driving the dynamic no longer adheres to the definition of a conventional subgradient operator and cannot be linked to any $L^q-$minimizing energy process (despite   quadratic $\beta$). However, we show here formally that, under specific conditions, this operator may be interpreted as a metric subgradient operator within the dual Sobolev space $W^{-1,p'}_{\Gamma_D}(\Omega)$ equipped with its natural distance.

   Assume   that 
  $$F(x,A)=\frac{1}{p'}\vert A\vert^{p'},\quad \hbox{ for any }A\in \RR^N,$$
  with $1<p'<\infty,$ 
and moreover    we assume that $g\equiv 0$  and $ \pi\equiv 0 .$  It is   clear in this case that the application     
    	$$f:= f_0+\nabla \cdot \overline f\in W^{-1,p'}_{\Gamma_D}(\Omega) \rightarrow \I_{0,0}^{\overline f}( f_0)^{1/p'}, $$
    	coincides with the usual  norm in $W^{-1,p'}_{\Gamma_D}(\Omega).$   Yet, one sees that in general the application $\I_{\pi,g}^\chi$ is not connected to a norm, unless one assumes some additional  assumptions  on $F$.  Thanks to Theorem \ref{tduality1}, for any $f\in W^{-1,p'}_{\Gamma_D}(\Omega),$ there exists $\phi_f\in   L^{p'}(\Omega)^N,$ given by 
    		$$\phi_f=\argmin \left\{\frac{1}{p'}\int_\Omega  \vert \phi\vert^{p'} \: :\:  \phi\in L^{p'}(\Omega)^N,\: -\nabla \cdot \phi=f_0+\nabla \cdot \overline f \hbox{ in }\Omega \hbox{ and }(\phi+\overline f)\cdot \nu  =0 \hbox{ on }\Gamma_N  \right\}  , $$ 
    	such that 
    	$$	\Vert f\Vert_{W^{-1,p'}_{\Gamma_D}(\Omega) } ^{p'}=  \frac{1}{p'}\int_\Omega  \vert \phi_f\vert^{p'}\: dx ,$$  
and 
$$\phi_f=\vert \nabla u\vert^{p-2}\nabla u \quad \hbox{ a.e. in }\Omega, $$
where   $u$ is the solution of the $p-$Laplacian equation with mixed boundary condition 
\begin{equation}
	\left\{ \begin{array}{ll}
		-\Delta_p u = f_0+ \nabla \cdot \overline f \quad & \hbox{ in }\Omega\\  \\ 
		(\phi+\overline f)\cdot \nu =0 &\hbox{ on }\Gamma_N\\ \\
		u=0 & \hbox{ on }\Gamma_D.	\end{array} \right.
\end{equation}
	In addition, one  can prove that the duality bracket $	\langle . , .\rangle_{\Omega}$ coincides with  
\begin{equation}
	\langle f , z\rangle_{\Omega}=\int_\Omega  \phi_f\cdot \nabla z \: dx  \quad \hbox{ for any } f\in W^{-1,p'}_{\Gamma_D}(\Omega)\hbox{ and }z\in W^{1,p}_{\Gamma_D}(\Omega).
\end{equation}
It is clear that if $p'=2,$ $W^{-1,2}_{\Gamma_D}(\Omega)$ is Hilbert space and  the approach we develop in this paper   may be   connected to the  earlier works   for quadratic $F$, \cite{Dam,DK,Kenmochi},     in the contest of       gradient flow   in $H^{-1}-$Hilbert spaces (one can see also the books \cite{Br,Barbu} for the general theory).         For the case where $p\neq 2,$   the natural duality bracket  $\langle . , .\rangle_{\Omega}$  do not allow to treat   doubly nonlinear PDEs   \eqref{FVnull}  in the contest of gradient flow.  To introduce   $W^{-1,p'}_{\Gamma_D}(\Omega)$-Gradient flow for \eqref{FVnull}, we consider the metric framework  where     $\X:=  W^{-1,p'}_{\Gamma_D}(\Omega)$  is equipped with its natural distance 
     	$$d (f_1,f_2)=   \Vert f_1-f_2\Vert_{W^{-1,p'}_{\Gamma_D}(\Omega) }.$$  
Then, under the assumptions of Section \ref{SAssumptions},  we consider the application  (internal energy)  $\mathcal E\: :\:   \X\to (-\infty,+\infty]$ given by 
$$ \mathcal E (\rho)=\left\{  \begin{array}{ll} 
	\int_\Omega  \beta(.,\rho)\: dx \quad &\hbox{ if }\beta(.,\rho)\in L^1(\Omega)\\ 
	+\infty & \hbox{ otherwise}.\end{array}\right. $$
	Revisiting the   steepest distance algorithm for the PDE \eqref{PDE1} with $f\equiv 0,$     we consider again   the time $\tau-$discretization $t_0 = 0 < t_1 < \cdots < t_n =T,$ with $t_i-t_{i-1}=\tau>0,$ for any $i=1,...n.$  Remember that    the    piecewise constant curves  \eqref{rhotau} is given by 
		\begin{equation} 
		\rho^i = 	 	\hbox{argmin}_{\rho\in   L^{p'}(\Omega)  }\left \{\int_\Omega   \beta(.,\rho) +   \tau \: \I_{0,0}^{0}   \left(\frac{  \rho^{i-1}-\rho}{\tau}\right)     \right\} .\end{equation}   
   Then, it is clear that 
		\begin{equation}\label{coroptim11}
		\rho^i = 	 	\hbox{argmin}_{\rho\in  X}\left \{\mathcal E (\rho) +  \tau^{p'-1}  d(   \rho^i ,\rho)^{p'}   \right\},
	\end{equation}    
	 which corresponds to the $p'-$scheme for the approximation of   $p'-$curves of maximal slope for $\mathcal E$ (cf. \cite{AGS}).  So, letting $\tau\to 0$ in the discrete curves 
 	\begin{equation} 
		\rho^\tau =\sum_{i=1}^n\rho^i\:  \chi_{]t_{i-1}  ,t_{i}]} \quad\hbox{ with } \rho^\tau(0)=\rho_0 ,\end{equation}    
we cover the $p'-$curve  of maximal slope for the   energy $\mathcal E.$    This implies in particular that the PDE \eqref{FVnull}   may be interpreted  as  a ''metric'' gradient flow  of $\mathcal E$ in $(X,d)$ ; i.e. 
 \begin{equation}\label{evolsubdiffW}
	\frac{d\rho}{dt}+ \partial_{\X}\mathcal E (\rho)= 0\quad  \hbox{ in }(0,T),
\end{equation}
in $(\X,d).$ The solution  may be interpreted in turn as the $p'-$curve  of maximal slope in $(\X,d)$ for the   energy $\mathcal E$ (cf. \cite{AGS}).   That is, $\rho$ satisfies the $p'-$energy identity 
\begin{equation}\label{pcurves}
 \frac{d}{dt } \mathcal E(\rho) +  	\frac{1}{p'}  \vert \rho'\vert^p +\frac{1}{p}  \vert \partial\mathcal E \vert^p (\rho)  = 0  \quad \hbox{ in }\D'([0,T]), \end{equation} 
where  $ \vert \partial  \mathcal E\vert  $ denotes the slopes of $\mathcal E$ and  $\vert \rho'\vert$ the metric derivative of $\rho.$ 
 In particular, Theorem \ref{texistVnull} gives an  interpretation in terms of PDE of the  $p'-$curve  of maximal slope given by 
 \eqref{pcurves}, and vice versa.  Moreover, one can associate the stationary PDE 
 	\begin{equation} 
 	\left\{\begin{array}{ll} 
 	  - \Delta_p \eta =  f    ,\quad       	\rho  \in \partial \beta^*(x,\eta  )    	\quad & \hbox{ in } \Omega, \\  \\
 		\vert \nabla \eta\vert^{p-2}\nabla \eta    \cdot \nu  =  0    & \hbox{ on }\Gamma_{N  }\\  \\
 		\eta   =  0    \quad    & \hbox{ on }\Gamma_{D  }  
 	\end{array}
 	\right. 	 
 \end{equation} 
 to the metric gradient flow of the internal energy $\int_\Omega   \beta(.,\rho) $ in $W^{-1,p'}_{\Gamma_D}(\Omega).$  
It is worth noting that, to the best of our knowledge, this type of results  is well-established for quadratic $F$ with $p=2$ and homogeneous boundary,    but has not been addressed for other cases. We believe that the approach developed in this paper could shed more light on    the $p'-$curve  of maximal slope solutions for such cases.   We do not go into depth on this direction in order to avoid making the paper much more  longer.

  	 \subsection{Prediction-correction algorithm}\label{SpredictionCor} 
  	 As we said in the introduction the     dynamic in  \eqref{PDEevol}   results  from the superposition of two processes  : a transport process  and a nonlinear energy minimization process.  The transport  process represents the external forces acting on the system, while the nonlinear energy minimization process represents the system's internal dynamics.   Remember that the Wasserstein steepest descent framework enables to handle the transport process within the internal energy and keeps the transition work out of   it .   For instance,  in the case where   	 where $V=\nabla \varphi,$   the solution of  diffusion-transport porous medium equation 
  	 \begin{equation}\label{PME2} 
  	 	\partial_t \rho -\nabla \cdot ( \rho^m -\rho\: V)=0, 
  	 \end{equation} 
  	  can  be obtained through  $W_2-$Wasserstein steepest descent framework, by working with $W_\tau (\rho,\tilde \rho)= \frac{1}{2\tau} \mathcal W_2( \rho,\tilde \rho)^2,$  and  employing the internal energy 
  	 $$ \mathcal E(\rho)=\frac{1}{m-1}\int_\Omega  \rho^{m} \: dx + \int_\Omega  \rho \: \varphi \: dx.$$
   In our approach, the transport process is incorporated within the traffic cost entity $\I$. Indeed, we retain the same internal energy $  \mathcal E(\rho)=\frac{1}{m+1}\int_\Omega  \rho^{m+1} \: dx$   as in the case without a transport $V,$  but the transition cost is modified to a $\chi-$forced one $\I^\chi,$ where $\chi$ involves $\rho\: V.$   Remember that $\I$ involves also boundary conditions $g$ and $\pi$ which can be considered also as boundary forced terms processes. 

%
  
  \medskip 	 
An alternative  way to address the diffusion-transport dynamics in Equation \eqref{PDEevol} which may be used successfully   is connected to the so-called prediction-correction algorithm (see some variants in \cite{MRS1,EIJ} in the contest of crowed motion and the references therein).   The idea  consists initially separates the dynamic into its primary processes: transport and diffusion.  In the first step, the transport process is applied to the system through  the transport equation $\partial_t \rho+\nabla \cdot (\rho\: V)=0.$ And in the second step,  one can address the diffusion process.  For this second step, we can use the proximal minimizing internal energy process  as given in Section \eqref{SProximal}  within a  transition cost  $\I$ 
 kick without  the forcing term $V.$

  More precisely, let us see how the algorithm turns   on formally for the evolution dynamic  \eqref{PDEevol}.  Again, we consider $T>0$ a given time horizon, and for a given time step $\tau >0,$ we consider a uniform partition   of  $[0,T]$ given by
  $t_k=k\tau,$  $k=0,\dotsc,n-1.$ Supposing that we know the density of the population $\rho_k$ at a given  step $k,$  starting by $\rho_0.$ Then, we superimpose successively the following two steps :
  \begin{itemize}
  	\item[]   \underline{Prediction:}
  	In the predictive time step $[t_k,t_{k+\frac{1}{2}}[,$ where $t_{k+\frac{1}{2}} =k\tau +\tau/2,$  the density of population trends to move  into
  	$$ \rho^{k+\frac{1}{2}} =\rho(t_{k+\frac{1}{2}}) ,$$
  	where $\rho$ is the solution of the transport equation 
  	\begin{equation}
  		\label{eq:continuity} 
  		\partial_t \rho+\nabla \cdot (\rho\: V )=0\quad \hbox{ in }[t_k,t_{k+\frac{1}{2}}),
  	\end{equation}
  	with  $ \rho(t_k)=\rho^k.$    
  	
  	\item[]\underline{Correction:} In general it is not expected that $\rho_{k+\frac{1}{2}}$ to be an allowable since the forcing term may increase the internal energy.  Then,   one can proceed minimizing instantaneously the energy  against the transition work $\I_{\pi,g}^\chi,$ with $\chi\equiv \tau \overline f^k$ (keeping $\chi$ out of $V$).     More precisely, one considers $\rho_{k+1}$ given by the following optimization problem
   \begin{equation}\label{argmincor}
  	\rho_{k+1} = 	 	\hbox{argmin}_{\rho}\left \{\int_\Omega   \beta(.,\rho) +  \tau\: \int_\Omega F(.,\phi/\tau ) \: dx - \tau\:   \langle \phi  \cdot  \nu, g\rangle_{\Gamma_D} \right. $$ $$\left.  \: :\:   \phi  \in  \A_{\pi}^{\tau \overline f^k}(\tau f_0^k +\rho^{k+1/2}-\rho],\:  \rho\in   L^{p'}(\Omega)  \right\}.
  \end{equation}   
    \end{itemize}	 
  Thanks to Theorem \ref{tduality2}, one sees that $	\rho_{k+1}$ is the solution of the PDE 
  	\begin{equation}
  	\label{PDEannexe}
  	\left\{\begin{array}{ll} 
  		\left.  \begin{array}{ll} 
  			\rho^{k+1}  -\tau\: \nabla\cdot ( \phi ^{k+1} +\overline f^k ) = \tau f^k_0 + \rho^{k+1/2} ,\quad     \\ 
  			\\  				\rho^{k+1} \in \partial \beta^*(x, \eta^{k+1} ),\quad \phi ^{k+1}   \in  \partial_\xi {F}^*(., \nabla \eta^{k+1})   \\   \\     	 
  		\end{array}\right\} 	\quad & \hbox{ in } \Omega, \\  \\
  		(\phi ^{k+1} +\overline f^k ) \cdot \nu  =  \pi    & \hbox{ on }\Gamma_{N  }\\  \\
  		\eta^{k+1}  =  g   ,\quad    & \hbox{ on }\Gamma_{D  }.
  	\end{array} \right.
  \end{equation} 
  
  \medskip 
  As in the algorithm of Section \eqref{SSteepest}, the main future is to prove that the  approximation 
    \begin{equation}\label{rhotaucor}
  	\rho^\tau =\sum_{k=0}^{N-1}\rho^{k+1/2}\:  \chi_{]t_{k}  ,t_{k+1/2}]} +  \rho^{k+1}\:  \chi_{]t_{k+1/2}  ,t_{k+1}]}  \quad\hbox{ with } \rho^\tau(0)=\rho_0,\end{equation}    
  converges to the solution of  \eqref{PDEevol}.  This type of algorithm is used in \cite{EIJ} for the study of crowd motion. 
  Other version based on $W_2-$Wasserstein steepest decsent algorithm may be found in \cite{MRS1,MRS2,MRSV,MS}.

     Both the algorithm presented in Section \ref{SSteepest} and the prediction-correction algorithm differ slightly in their approach. Consider, for instance, a system of particles in a fluid, where the particles are driven by the fluid's velocity and simultaneously attempt to minimize their energy by transitioning to a lower energy state. While the first algorithm focuses on minimizing the energy state by incorporating the force term into the transition cost, the second algorithm aims to first propel the system with the fluid's velocity and then guide it towards a lower energy state using a diffusive transition cost. In essence, the second algorithm employs a prediction-correction strategy, whereas the first algorithm can be considered a "transport forcing transition" type.
 
At last, it is worth mentioning that we employ a "transport forcing transition" type algorithm in this paper to address Equation \eqref{PDEevol} due to its ability to simultaneously recover compactness and effectively enforce Neumann boundary conditions consistent with those of \eqref{PDEevol} without the need for additional constraints on V at the boundary. Otherwise, careful consideration must be given to the Neumann boundary condition when utilizing the prediction-correction algorithm described above. 

\subsection{Uniqueness} 	The uniqueness of the solution for the PDE \eqref{PDEevol} is a very interesting question that can quickly become very complicated depending on the assumptions on $\beta,$ $F$ and $V.$ In the case where $V\equiv 0,$ one might expect the uniqueness   either through convexity arguments for the optimization problem or by using PDE techniques to prove   $L^1-$contraction principle for the solutions. However, this is no longer true for more general  $V$ and the question of uniqueness becomes challenging  even for some concrete choice of $F$ or $\beta$ (see for instance Remark 2 in \cite{IgShaw}) with respect to the assumptions on $V$ as well as the choice of boundary conditions.     To the best of our knowledge, except   particular cases involving linear type diffusion (cf. \cite{IgShaw,KM,David,MS,Car,IgPME} for linear diffusion and \cite{IgUr2} for the $p-$Laplacian operator) , the question of uniqueness of weak solutions  with nonlinear diffusion and   general non linearity  is not  addressed   in the literature. We will not treat  these questions in this paper.

\subsection{Linear growth  cost  $F$ : granular diffusion}	The case where  $F$ satisfies  $(H1)$ with $p'=1 $  and  $p=\infty,$ appears  when one deals with granular diffusion phenomena like in sandpile or crowd motion (cf. \cite{EIJ}). In this case, one needs  to use vector valued fields   whose divergence  are Radon measures (cf.  \cite{Chen}).   One can see the papers \cite{EIN1,EIN2}  and \cite{IN1,IN2}  for the study of some particular cases connected to $I_{\pi,g}^\chi.$  The steepest descent algorithm of the type $\N$ has been used in \cite{EIJ}     with linear growth  $F$ and Hele-Shaw type nonlinearity to model  congestion phenomena in crowed motion.  Formally, the evolution problem of the type \eqref{PDEevol} associated with linear growth  $F,$ may be  given by 
 \begin{equation} 
	\left\{\begin{array}{ll} 
		\left.  \begin{array}{ll} 
			\partial_t 	\rho  - \nabla\cdot ( \phi   +\rho  V )=  f    ,\quad        \\ 
			\\  		\rho \in \partial \beta^*(x,\eta  ),\:   \phi    \in   \partial I\!\! I_{B_F} (\nabla \eta )    \\   \\   
		\end{array}\right\} 	\quad & \hbox{ in } Q, \\  
		(\phi   +\rho  V )   \cdot \nu  =  \pi    & \hbox{ on }\Sigma_{N  }\\  \\
		\eta   =  g    \quad    & \hbox{ on }\Sigma_{D  } \\  \\
		\rho(0)=\rho_0& \hbox{ in }\Omega ,
	\end{array}
	\right. 	 
\end{equation} 
 where 
 $$B_{F(x,.)}:=\left\{ A\in \RR^N\: :\:   F^*(x,A)\leq 1\right\} , \quad \hbox{ for a.e. }x\in \Omega.$$
For instance, if $F(x,A)=k(x)\: \vert A\vert,$ for any $x\in \Omega$ and $A\in \RR^N,$  $B_{F(x,.)}$ coincides with the   closed bull $\overline{B(0_{\RR^N},k(x))}.$ Yet, to be well posed one needs to assume some necessary conditions  on $g$ and $\pi.$     Moreover, the involved flux $\phi $  is no more  a Lebesgue  integrable   function  (see   \cite{Dweik} and the references therein for  Lebesgue regularity issues of such flux $\phi $). It is a Radon measure and one needs to use more involved techniques to treat this case.  We will back on this case on forthcoming works.


\subsection{Linear growth  cost  $F^*$ : total variation flow }	    
 A common scenario that we do not address in this paper corresponds to $F$ such that $F^*$ has linear growth function (cf. \cite{GM}).  This case arises in instances like total variation flow (cf. \cite{ACMazon}) and doubly nonlinear total variation flow (cf. \cite{MMT}), where $p=1.$ In this case, the functional framework outlined in Section \ref{SProximal} is no longer applicable. Consequently, a thorough analysis of this scenario necessitates a careful examination in the appropriate functional spaces.

\subsection{Numerical computation}	
Gradient descent algorithms are a powerful class of iterative optimization methods that leverage the gradient of a function to guide the search for its optimal solution. While they are typically employed for nonlinear optimization problems, as we demonstrate in Section \ref{SSteepest}, they can also effectively approximate the solutions of  PDEs. This algorithmic approach has gained prominence in the context of Wasserstein optimization, particularly through the JKO algorithm, a cornerstone of numerical approximation in this framework. For a comprehensive understanding of these methods and their numerical considerations in this contest, we refer the reader to the works \cite{BCL,BCCNP} and Chapter 6 of the book \cite{Stbook}.

To our knowledge, the recent contribution \cite{EIJ} pioneers the use of steepest descent through minimum flow to numerically tackle evolution dynamics of the form \eqref{PDEevol}. The authors highlight the algorithm's adaptability in    handling both quadratic and linear growth cost $F,$ incorporating Hele-Shaw type nonlinearities to model congestion phenomena in crowd motion. Moreover, the algorithm's flexibility enables the incorporation of diverse practical scenarios as well as mixed and non-homogeneous boundary conditions, paving the way for more realistic simulations.

So,  basically the steepest descent through minimum flow algorithm introduced in this paper presents a novel approach for numerically addressing a broad spectrum of composite PDEs. By utilizing sophisticated numerical optimization techniques known for their effectiveness, even in the presence of low-regularity nonlinearities, this algorithm opens up new methods for PDE analysis and simulation. Specifically, for numerical approximation of the solution of \eqref{PDEevol}, Euler-Implicit discretization proves to be a valuable tool. This leads to the numerical solution of the PDE \eqref{PDEi}. By connecting this PDE to the steepest descent formulation of the type $\N $ and exploiting the duality between these problems, we gain access to numerical convex optimization techniques, particularly primal-dual methods such as the alternating direction method of multipliers (ADMM), the augmented Lagrangian method (cf. \cite{Glowinski}), and the Chambolle-Pock algorithm (cf. \cite{ChP}). These methods offer a robust framework for handling the challenges of PDE approximation, particularly when dealing with complex nonlinearities and non-homogeneous boundary conditions.

 \appendix

  	  	\section{Appendix : technical tools}\label{SAppendix} 
  	 \setcounter{equation}{0}

  	\subsection{On $W^{-1,p'}_{\Gamma_D}(\Omega)$ characterization} \label{AppendixW}
  	The aim of the following Lemma is to give a practical  characterization  of the dual bracket $\langle.,.\rangle_{W^{-1,p'}_{\Gamma_D}(\Omega),W^{1,p}_{\Gamma_D}(\Omega)} ,$ that we denote by   $\langle.,.\rangle_{ \Omega} ,$ in terms of PDE.  To this aim,  we prove a decomposition of elements of  $W^{-1,p'}_{\Gamma_D}(\Omega)$ through a divergence formulation  that includes a trace of these elements on $\Gamma_N.$ More involved proofs for this type of results that remain valid for a more general class of $\Omega$ can be found in  \cite{LionsMagenes} for $p=2$ and  \cite{Cherif2} for general $1<p<\infty.$ 
  	
  	\begin{aplemma} \label{LdualW}
  		Under the assumptions of Section  \ref{SAssumptions}, $f\in  W^{-1,p'}_{\Gamma_D}(\Omega)$ if and only of one the following equivalent formulation is fulfilled
  		\begin{itemize}
  			\item[(A1)] there exists $f_0\in L^{p'}(\Omega),$ $\overline f\in L^{p'}(\Omega)^N,$ such that 
  			\begin{equation} \label{W1}
  				\langle f,\xi\rangle_{\Omega}  = \int_\Omega f_0\:\xi\: dx -   \int_\Omega \overline f\cdot \nabla \xi  \: dx ,\quad \forall\xi\in 
  				W^{1,p}_{\Gamma_D}(\Omega). 
  			\end{equation}
  			
  			\item[(A2)] there exists $f_0\in L^{p'}(\Omega),$ $\tilde  f\in L^{p'}(\Omega)^N$ and $f_{\Gamma_N} \in W^{-\frac{1}{p'},p'}(\Gamma_N),$  such that 
  			\begin{equation} \label{W2}
  				\langle f,\xi\rangle_{\Omega}  = \int_\Omega f_0\:\xi\: dx -   \int_\Omega \tilde  f\cdot \nabla \xi  \: dx  + \langle f_{\Gamma_N},\xi\rangle_{\Gamma_N} ,\quad \forall \xi\in 
  				W^{1,p}_{\Gamma_D}(\Omega). 
  			\end{equation}
  		\end{itemize}  
  	\end{aplemma}
  	\begin{proof} It is clear that for a given $f_0\in L^{p'}(\Omega)$ and $\overline f\in L^{p'}(\Omega)^N,$ if $f$ satisfies   \eqref{W1}  then  $f\in  W^{-1,p'}_{\Gamma_D}(\Omega).$ Similarly, for a given $f_0\in L^{p'}(\Omega),$ $\tilde  f\in L^{p'}(\Omega)^N$ and $f_{\Gamma_N} \in W^{-1/p',p'}(\Gamma_N),$ if $f$ satisfies   \eqref{W2}  then  $f\in  W^{-1,p'}_{\Gamma_D}(\Omega).$   
  		To prove that (A1)  is a necessary condition,  for a given $f\in W^{-1,p'}_{\Gamma_D}(\Omega),$    we consider the 
  		the minimization problem 
  		\begin{equation}\label{plaplace2}
  			\min_{z\in W^{1,p}_{\Gamma_D}(\Omega)} \left\{ \frac{1}{p}\int _\Omega \vert \nabla z\vert^p  \: dx  - 	\langle f,z\rangle_{\Omega}  \right\}  \end{equation}
  		Working as  in Lemma \ref{Fnonempty}, 
  		we can prove that    this problem has  
  		a solution $z $ which   satisfies 
  		\begin{equation}\label{divphif}
  			\int _\Omega  \vert \nabla z\vert^{p-2}\nabla z \cdot \nabla \xi = \langle f,\xi\rangle_{\Omega} ,\quad \forall\: \xi\in  W^{1,p}_{\Gamma_D}(\Omega).  
  		\end{equation}
  		Then,  we deduce that  (A1) is fulfilled  by taking simply $f_0=f$ and $\overline f=0$ whenever $f\in L^{p'}(\Omega),$  otherwise   $f_0=0$ and $\overline f= \vert \nabla z\vert^{p-2}\nabla z.$

  		Now let us prove that (A2) is also a necessary condition.  Since $f\in W^{-1,p'}_{\Gamma_D}(\Omega),$ we know that there exists   $f_0\in L^{p'}(\Omega)$ and $\tilde  f\in L^{p'}(\Omega)^N$ such that $f$ coincides with $f_0+\nabla \cdot \tilde f$ in the dual Sobolev space $(W^{1,p}_0(\Omega))^*$ ; i.e.
  		\begin{equation}
  			\langle f,\xi\rangle_{\Omega}     = \int_\Omega f_0\:\xi\: dx -   \int_\Omega \tilde  f\cdot \nabla \xi  \: dx ,\quad\forall\:  \xi\in 
  			W^{1,p}_{0}(\Omega). 
  		\end{equation}
  		Keeping in mind  that  $f$ satisfies \eqref{plaplace2},  we have  
  		$$\int_\Omega  \phi \cdot \nabla \xi \: dx  = \langle f,z\rangle_{\Omega}  =    \int _\Omega f_0\: \xi \: dx  -\int_\Omega  \tilde f\cdot \nabla \xi \: dx, \quad \forall\: \xi\in  W^{1,p}_{0}(\Omega),$$ 
  		so that  that $\nabla \cdot (\phi+\tilde f)=f_0$ in $\D'(\Omega).$ 
  		In particular, this implies that  $\phi+\tilde f \in  H^{p'}(div,\Omega),$       $(\phi+\tilde f ) \cdot \nu  \in W^{-\frac{1}{p'},p'}(\Gamma_N)$ is well defined, and we have 
  		\begin{equation} 
  			\int_\Omega (\phi+\tilde f ) \cdot \nabla \xi   \: dx =   \int_\Omega f_0\: \xi \: dx  + 	\langle (\phi+\tilde f ) \cdot \nu  ,\xi   \rangle_{ \Gamma_N}  ,  \quad  \forall\: \xi\in  W^{1,p}_{\Gamma_D}(\Omega). 
  		\end{equation}  
  		Thus \eqref{W2} is fulfilled by taking   $f_{\Gamma_N}:= (\phi+\tilde f ) \cdot \nu.$

  	\end{proof}
  	
  	\begin{apremark}
  		\begin{enumerate}
  			\item   Lemma \ref{LdualW} aims to give a practical  characterization  of the dual bracket $\langle.,.\rangle_{W^{-1,p'}_{\Gamma_D}(\Omega),W^{1,p}_{\Gamma_D}(\Omega)} $   in terms of PDE. This characterizations emerges a contribution  of any element $f\in W^{-1,p'}_{\Gamma_D}(\Omega)$ on the boundary $\Gamma_N.$  In particular, Lemma \ref{LdualW}  presents  two equivalent manners to handle this contribution. 
  			\begin{enumerate}
  				\item In (A1), this contribution is linked to the normal trace of $\overline f$ in \eqref{W1},   if it ever exists in a some ''good sense''.  
  				Indeed, solving $- \nabla \cdot \phi=f\in W^{-1,p'}_{\Gamma_D}(\Omega)$  is equivalent to  find $\phi$ such that  the PDE  
  				$$\left\{ \begin{array}{ll}  
  					\nabla \cdot \phi=f_0+\nabla \cdot \overline f \quad  & \hbox{ in }  \Omega \\  \\ 
  					(\phi+\overline f)\cdot \nu =0 &\hbox{ on }\Gamma_D\end{array} \right.$$
  				is fulfilled  in a weak sense ; i.e. 
  				$$\int_\Omega  (\phi+\overline f)\cdot \nabla \xi =\int _\Omega f_0\: \xi,\quad \forall\:  \xi\in W^{1,p}_{\Gamma_D}(\Omega).$$  In some sense the boundary condition $(\phi+\overline f)\cdot \nu$ needs to be understood in $W^{-1/p',p'}(\Gamma_N)$. Yet, despite particular cases,  the normal trace of $\overline f$ is not $W^{-1/p',p'}(\Gamma_N)$ in general.

  				\item  The characterization (A2) aims  to exhibit explicitly a  $W^{-1/p',p'}(\Gamma_N)$   contribution of $f$ on the boundary $\Gamma_N$.  Yet, one sees that the formulation still involve the normal boundary trace of $\tilde f$ which is not $W^{-1/p',p'}(\Gamma_N)$  in general.   See that,  solving $- \nabla \cdot \phi=f\in W^{-1,p'}_{\Gamma_D}(\Omega)$ with this formulation is equivalent  to  find $\phi$ such that  the PDE  
  				$$\left\{ \begin{array}{ll}  
  					\nabla \cdot \phi=f_0+\nabla \cdot \tilde  f \quad  & \hbox{ in }  \Omega \\  \\ 
  					(\phi+\tilde f)\cdot \nu =\pi &\hbox{ on }\Gamma_D\end{array} \right.$$
  				in a weak sense ; i.e. 
  				$$\int_\Omega  (\phi+\overline f)\cdot \nabla \xi =\int_\Omega  f_0\: \xi +  	\langle \pi ,\xi   \rangle_{ \Gamma_N} ,\quad  \forall\: \xi\in W^{1,p}_{\Gamma_D}(\Omega).$$  Again,  the boundary condition $(\phi+\overline f)\cdot \nu$ needs to be understood in $W^{-1/p',p'}(\Gamma_N)$.    In some sense, the formulation (A2) aims just to rewrite $\overline f$ by emerging explicitly an   $W^{-1/p',p'}(\Gamma_N)$ term. This makes it easier to write $f$ in situations where $\overline f$  can be expressed clearly using a $\tilde f,$ which has a  null normal trace on the boundary $\Gamma_N,$ and a   $W^{-1/p',p'}(\Gamma_N)$ term.   Yet, this is not fulfilled  in general.  
  				
  			\end{enumerate} 
  			
  			\item As we see in the proof, the term $f_0$ may be considered equals to $0.$  Nevertheless, we keep this term in the decomposition to make it easier to write $f$ in situations where $  f$ is an $L^{p'}$ function.  
  		\end{enumerate}
  		
  	\end{apremark}

  	\subsection{PDE chain rule formula}
  	
  	\begin{aplemma}\label{ltech2}
  		Let $\rho,\: f \in L^{p'}(0,T;W^{-1,p'}_{\Gamma_D}(\Omega)) $  and $\Psi\in L^{p'}(Q)$ be such that  $ \partial_t 	 \rho  =   \nabla\cdot  \Psi  + f     $ in $Q$ and $	\Psi  \cdot \nu  =  \pi $ on $\Sigma_{N  }$ in the sense that 
	\begin{equation}  
	\frac{d}{dt}\int_\Omega \rho\: \xi\: dx + \int _\Omega \  \Psi    \cdot \nabla \xi \: dx = \langle f,\xi\rangle_{\Omega} + \langle \pi,\xi\rangle_{\Gamma_N} ,\quad \hbox{ in }\D'([0,T)),\quad \forall \xi\in W^{1,p}_{\Gamma_D}(\Omega).
\end{equation}   
  		Then, for any $\eta \in L^p(0,T;W^{1,p}(\Omega))$ such that $\eta =g$ on $\Gamma_D$ and $\eta\in \partial \beta(.,\rho)$ a.e. in $Q,$ we have 
  		\begin{equation}\label{integraleEq}
  			\frac{d}{dt}\int_\Omega (\beta(x,\rho) -\tilde g\: \rho)  +\int_\Omega \Psi\cdot \nabla (\eta-\tilde g) = \langle f, \eta-\tilde g\rangle_{\Omega}  + \langle \pi,\eta \rangle_{\Gamma_N}, \quad \hbox{ in }\D'([0,T))   . 
  		\end{equation}
  	\end{aplemma}
  	\begin{proof} 
  		Let $\psi \in \mathcal{D}(]0,T[\times \mathbb{R}^N)$ be such that $\psi\ge 0.$ For     any $h>0$, we consider 
  		$$\eta^n_h(t):=\frac{1}{h}\int_t^{t+ h} \underbrace{(\eta(s) -\tilde g)}_{=:\eta_g(s)} \: \psi(s)ds,$$
  		 See that  $\eta^n_h$ can be used as test function in (\ref{integraleEq})
  		and therefore
  		$$ \underbrace{\int_0^T\int_\Omega \rho (t)\frac{\eta_g (t+h)\psi(t+h)-\eta_g(t)\psi(t)}{h}}_{=: I^n_h} 	+  \int_0^T\!\! \int_\Omega \Psi  \cdot \nabla \eta^n_h =  \int_0^T \langle  f, \eta^n_h \rangle_{\Omega}  + \int_0^T\!\! \langle \pi,\eta \rangle_{\Gamma_N}     .$$ 
  		We see that 	$$\begin{array}{ll} 
  			I^n_h   
  			&=  
  			\int_0^T\int_\Omega \rho (t)\frac{\eta (t+h)\psi(t+h)-\eta(t)\psi(t)}{h} 
  			-  \int_0^T\int_\Omega g\: \rho (t)\frac{ \psi(t+h)-\psi(t)}{h} 
  			\\ \\     	& 	=
  			\int_0^T\int_\Omega\frac{\rho(t-h)-\rho(t)}{h}\eta_g(t)\psi(t)   -  \int_0^T\int_\Omega g\: \rho (t)\frac{ \psi(t+h)-\psi(t)}{h}   \\  \\ 
  			&\leq \int_0^T\int_\Omega\frac{\beta(\rho(t-h))-\beta(\rho(t))}{h} \psi(t)    -  \int_0^T\int_\Omega g\: \rho (t)\frac{ \psi(t+h)-\psi(t)}{h} \\  \\ 
  			&\leq \int_0^T\int_\Omega\frac{\psi(t+h)   -\psi(t)  }{h}  \beta(\rho(t))  -  \int_0^T\int_\Omega g\: \rho (t)\frac{ \psi(t+h)-\psi(t)}{h}  .
  		\end{array}$$
  		So, taking limits as $\tau\to 0^+$ we get 
  		$$ \int_0^T\int_\Omega \partial^n_t \psi\:  (\beta(x,\rho(t))-\tilde g)   +  \int_0^T\!\! \int_\Omega  \Psi  \cdot \nabla (\eta-\tilde g) \: \psi  \geq   \int_0^T \psi \: \langle   f, \eta-\tilde g\rangle_{\Omega}     + \int_0^T\psi\:  \langle \pi,\eta \rangle_{\Gamma_N}     $$
  		Taking now
  		$\tilde\eta_\tau(t)=\frac{1}{\tau}\int_t^{t+\tau} \eta(s-\tau)\psi(s)ds$,
  		and arguing as above we get the converse  inequality. That is  
  		$$ \int_0^T\int_\Omega \partial^n_t \psi\:  (\beta(x,\rho(t))-\tilde g)   +  \int_0^T\!\! \int_\Omega \Psi  \cdot \nabla (\eta-\tilde g) \: \psi  =   \int_0^T\psi \:  \langle   f, \eta-\tilde g\rangle_{\Omega}    + \int_0^T  \langle \pi,\eta\rangle_{\Gamma_N}   .  $$ 
  		Thus \eqref{integraleEq}. 
  	\end{proof}

  	\subsection{Week  Aubin's result}
  	\begin{aptheorem}[cf. \cite{ACM} and \cite{Moussa}]\label{ThBCM}
  		Let $(u_n)_n$ and $(v_n)_n$ be two bounded sequences in $L^p (0,T;  W^{1,p} (\Omega) ) $  and $L^{p'}(Q),$ respectively. Assume moreover that   $v_n$ satisfies the following weak estimates 	\begin{equation}
  			\int_0^T\!\!\int_\Omega v_n\: \partial^n_t\:  \varphi\: dxdt \leq C\Vert \nabla \varphi\Vert_{L^\infty(Q)},\quad \hbox{ for any }\varphi\in\D( Q). 
  		\end{equation} 
  		If  $u\in L^p (0,T;  W^{1,p} (\Omega) ) $  and $v\in L^{p'}(Q),$ are such that, up to extraction of a subsequence,  $$ u_n\to u   \hbox{ in }L^p (0,T;  W^{1,p} (\Omega) ) -\hbox{weak},  \hbox{ and }  v_n\to v   \hbox{ in }L^{p'} (Q) -\hbox{weak}  $$
  		then 
  		\begin{equation}
  			\int_0^T\!\!\int_\Omega u_nv_n\: \varphi\: dxdt \to 	\int_0^T\!\!\int_\Omega u\: v \: \varphi\: dxdt  \quad \hbox{ for any }\varphi\in \D ( Q).
  		\end{equation}
  		If moreover, $v_n\in \alpha(.,u_n)$ a.e. in $Q,$ where  $\alpha(z,.)_{z\in Q}$ is a sequence of measurable maximal monotone graph    with $0\in \alpha(.,0),$ a.e. in $Q,$ then  the following assertions are fulfilled \begin{enumerate}
  			\item $v\in \alpha(.,u)$ a.e. in $Q,$
  			\item a.e. $(t,x)\in Q,$ either  $u_n(t,x)\to u(t,x)$ or $v_n(t,x)\to v(t,x),$ 
  			\item  if $\alpha$  (resp. $\alpha^{-1}$)  is single valued then  $v_n \to v $ (resp.  $u_n \to u$) a.e. in $Q.$
  			
  		\end{enumerate}
  	\end{aptheorem}
  	For the proof of this theorem one can see the papers \cite{ACM} and \cite{Moussa}.

  	\section*{Acknowledgments} 
  	We would like to thank Chérif Amrouche  for many  interesting discussions  about Sobolev dual spaces.

\end{document}